\documentclass{article}
\usepackage{amsmath,amsfonts,amsthm,amssymb,amscd,color,xcolor}
\usepackage{graphicx,amscd,mathrsfs,wrapfig,mathrsfs,lipsum}
\usepackage{microtype}
\usepackage{float}
\usepackage{multicol}
\usepackage{caption}
\usepackage{capt-of}
\usepackage{hyperref}

\colorlet{darkblue}{blue!50!black}

\hypersetup{
    colorlinks,%
    citecolor=darkblue,%
    filecolor=red,%
    linkcolor=darkblue,%
    urlcolor=blue,%
    pdfnewwindow=true,%
    pdfstartview={FitH}
}

\newcommand{\p}{\partial}
\newcommand{\e}{\varepsilon}

\newcommand{\R}{{\mathbb R}}
\newcommand{\IP}{{\mathbb P}}

\newcommand{\Z}{{\mathbb Z}}
\newcommand{\E}{{\mathbb E}}
\newcommand{\T}{{\mathbb T}}

\newcommand{\GGG}{\boldsymbol{\mathit G}}

\newcommand{\et}{\boldsymbol{\mathit e_3}}

\newcommand{\bT}{\bar{T}}

\newcommand{\BB}{{\mathcal B}}

\newcommand{\DD}{{\mathcal D}}
\newcommand{\EE}{{\mathcal E}}
\newcommand{\FF}{{\mathcal F}}

\newcommand{\HH}{{\mathcal H}}
\newcommand{\KK}{{\mathcal K}}
\newcommand{\LL}{{\mathcal L}}

\newcommand{\PP}{{\mathcal P}}
\newcommand{\RR}{{\mathcal R}}
\newcommand{\sS}{{\mathcal S}}

\newcommand{\XX}{{\mathcal X}}

\newcommand{\ZZ}{{\mathcal Z}}

\newcommand{\UU}{{\mathcal U}}

\newcommand{\dd}{{\textup d}}

\newcommand{\mmmm}{{\mathfrak m}}

\newcommand{\PPPP}{{\mathfrak P}}

\newcommand{\fff}{{\boldsymbol{\mathit f}}}

\newcommand{\uuu}{{\boldsymbol{\mathit u}}}
\newcommand{\vvv}{{\boldsymbol{\mathit v}}}
\newcommand{\www}{{\boldsymbol{\mathit w}}}
\newcommand{\twww}{{\boldsymbol{\widetilde{\mathit w}}}}

\newcommand{\yyy}{{\boldsymbol{\mathit y}}}
\newcommand{\zzz}{{\boldsymbol{\mathit z}}}

\newcommand{\Ra}{\mathrm{Ra}}
\newcommand{\Prn}{\mathrm{Pr}}

\newcommand{\supp}{\mathop{\rm supp}\nolimits}
\newcommand{\diver}{\mathop{\rm div}\nolimits}

\newcommand{\lspan}{\mathop{\rm span}}

\theoremstyle{plain}
\newtheorem*{mt}{Main Theorem}

\newtheorem{theorem}{Theorem}[section]
\newtheorem{lemma}[theorem]{Lemma}

\newtheorem{proposition}[theorem]{Proposition}
\newtheorem{corollary}[theorem]{Corollary}
\theoremstyle{definition}

\theoremstyle{remark}
\newtheorem{remark}[theorem]{Remark}

\newtheorem*{example*}{Example}

\numberwithin{equation}{section}

\begin{document}
\author{Juraj F\"oldes\footnote{University of Virginia, Department of Mathematics, Charlottesville, VA, 22904, USA; e-mail: \href{mailto:foldes@virginia.edu}{Foldes@virginia.edu}} \and  Armen Shirikyan\footnote{Department of Mathematics, CY Cergy Paris University, CNRS UMR 8088, 2 avenue Adolphe Chauvin, 95302 Cergy--Pontoise, France \& Peoples Friendship University of Russia (RUDN University); e-mail: \href{mailto:Armen.Shirikyan@cyu.fr}{Armen.Shirikyan@cyu.fr}}} 
\title{Rayleigh--B\'enard convection with stochastic  forcing localised near the bottom}
\date{\today}
\maketitle

\begin{abstract}
We prove stochastic stability of the three-dimensional  Rayleigh--B\'enard convection in the infinite Prandtl number regime for any pair of temperatures maintained on the top and the bottom. Assuming that the non-degenerate random perturbation acts in a thin layer adjacent to the bottom of the domain, we prove that the random flow periodic in the two infinite directions stabilises to a unique stationary measure, provided that there is at least one point accessible from any initial state. We also prove that the latter property is satisfied if the amplitude of the noise is sufficiently large. 

\medskip
\noindent
{\bf AMS subject classifications:} 35Q35, 76E06, 76M35, 93B07

\smallskip
\noindent
{\bf Keywords:} Boussinesq equation, Rayleigh--B\'enard convection, random forcing, exponential mixing, unique continuation
\end{abstract}

\tableofcontents

\section{Introduction}
\label{s0}
Buoyancy driven flows can be observed in various natural settings, ranging from the boiling water in a pot, through the motion of  oceans and atmosphere, to large scale flows inside of planets or stars. Convection is induced by gravity and the fact that colder fluid has higher density than the warmer one. Thus, if there is a  heat source on the bottom (in the direction of gravity vector), then the denser layer lies on the top of the lighter one  which induces an instability of the trivial state (no motion of the fluid) and leads  to convection. The standard model that already contains the essence of the general situation is given by Rayleigh--B\'enard convection (see \cite{benard1901, rayleigh-1916}): the fluid is enclosed between two parallel plates with the normal vector parallel to the gravity force. The fluid is heated from below (and possibly cooled from above) and for simplicity one often assumes periodic boundary conditions in the horizontal direction. 

Mathematically,   Rayleigh--B\'enard convection is modeled by Boussinesq equations \cite{boussinesq1897} which describe the evolution of the velocity~$\uuu$ and  temperature~$T$. After appropriate rescaling (see e.g.~\cite{FGRW-2016}), the Boussinesq equations take the form
\begin{align}
	\frac{1}{\Prn}(\uuu_t + \langle\uuu, \nabla\rangle \uuu) - \Delta \uuu +\nabla p&=(\Ra\,T)\et, \quad \diver \uuu=0 \,, 
	\label{full-vel}\\
	\p_t T- \Delta T+\langle \uuu,\nabla \rangle T&=\eta(t,x)\,,  \label{full-temp} 
\end{align}
where the non-dimensional parameters $\Ra$ and $\Prn$ are respectively Rayleigh and Prandtl numbers, $\et$ is the vertical basis vector, $T$, $\uuu$, and~$p$ are unknown temperature, velocity field, and pressure of the fluid, and~$\eta$ is an external heating source. System~\eqref{full-vel}, \eqref{full-temp} is considered in the domain
\begin{equation} \label{domain}
D=\{x=(x_1,x_2,x_3)\in\R^3: 0<x_3<1\},
\end{equation}
with the $2\pi$-periodicity condition with respect to the variables~$x_1$ and~$x_2$. The model is known to manifest strong instability and chaoticity phenomena (e.g., see~\cite{doering-atal-2006,mielke-1997}). On the other hand, it is widely believed that random external forces may contribute to the large-time stabilisation of the flow and to its convergence to a stationary steady state. This question was investigated in the papers \cite{FGRT-2015, FGRW-2016, FFGR-2017, FGR-2019}. In particular, it was proved that if the noise is non-degenerate in the sense that it acts on sufficiently many eignefunctions of the Laplacian, then the Markov process associated with~\eqref{full-vel}, \eqref{full-temp} has a unique stationary measure, and any other solution converges to it in distribution as $t\to\infty$. 

In many situations when the fluid is very viscous or has small thermal diffusivity (for example, earth mantle, high pressure gasses, or engine oil), the Prandtl number $\Prn$ can reach the order of~$10^{24}$, see \cite{CD-1999, DC-2001, OS-2011}. Then, it is usual to formally set $\Prn = \infty$ and approximate \eqref{full-vel}, \eqref{full-temp} by the coupled elliptic--parabolic system
\begin{align}
	\p_t T- \Delta T+\langle \uuu,\nabla \rangle T&=\eta(t,x),\label{temperature}\\
	-\Delta \uuu +\nabla p&=(\Ra\,T)\et, \quad \diver \uuu=0. \label{velocity-field}
\end{align}
Mathematically, this is an active scalar equation for the temperature~$T$, and the velocity~$\uuu (t)$ is completely determined by $T(t)$ for any time $t > 0$; in particular, there is no initial condition for $\uuu$ or an independent evolution equation. Note that the advection term $\langle\uuu, \nabla\rangle \uuu$ is not present in \eqref{velocity-field}, which allows us to treat the three-dimensional case, whereas for the full Boussinesq system the well-posedness is unknown even if $\eta \equiv 0$. Let us mention that the infinite Prandtl system \eqref{temperature}, \eqref{velocity-field}  can have very complicated dynamics (even without forcing) if the Rayleigh number~$\Ra$ is sufficiently large; see \cite{BH-2009, CD-1999, BPA-2000, DC-2001, wang-2004, park-2006, AGL-2009, LX-2010, OS-2011}. 

For viscous fluids, the no-slip condition $\uuu =0$ on the boundary is the most natural requirement for the velocity, although other boundary conditions have been considered in the literature. To  create the temperature gradient that generates buoyancy force and induces non-trivial dynamics, we assume the non-homogeneous Dirichlet boundary conditions for the temperature. We thus impose the conditions
\begin{equation} \label{boundary-condition}
	T\bigr|_{x_3=0}=T_b, \quad
	T\bigr|_{x_3=1}=T_u, \quad 
	\uuu\bigr|_{x_3=0}=\uuu\bigr|_{x_3=1}=0,
\end{equation}
where $T_b > T_u$ are some numbers. In more complicated situations, the temperature on the lower (and also upper) boundary is not constant, and the fluctuations are often modeled by random processes. In addition, the fluctuations might penetrate inside of the domain and act in a small layer near the boundary and as such the fluctuations are spatially correlated. For concrete examples arising in physics, we refer to models of  absorption of sunlight in frozen lakes~\cite{Farmer1975, bengtsson1996,toppaladoddi2018,ulloa2018}, interior of stars~\cite{kippenhahn2012,barker2014}, radioactive decay in Earth's mantle~\cite{davaille2002}, flux of neutrinos in collapsing stellar cores~\cite{janka1996}, and laboratory experiments~\cite{bouillaut2019}. This motivates our study of the stochastic forcing having a support in a boundary layer, and space-time correlation allows us to assume bounded forcing (or fluctuations). 
  
We thus consider the problem of stochastic stabilisation of the system~\eqref{temperature}, \eqref{velocity-field} in the strip~\eqref{domain}, imposing the $2\pi$-periodicity condition in the horizontal directions~$x_1$ and~$x_2$, and Dirichlet's condition~\eqref{boundary-condition} on the top and the bottom. To single out a unique solution of the problem, we specify an initial condition for~$T$:
\begin{equation} \label{IC-temperature}
	T(0,x)=T_0(x). 
\end{equation}
As for the external heating source, we assume that it is concentrated near the bottom and has the form
\begin{equation} \label{eta-form}
	\eta(t,x)=\sum_{k=1}^\infty I_{[k-1,k)}(t)\eta_k(t-k+1,x),
\end{equation}
where $I_{[k-1,k)}$ is the indicator function of the interval $[k-1,k)$, and~$\{\eta_k\}_{k\ge1}$ is a sequence of i.i.d.\ random variables in~$L^2([0,1]\times D)$. Under suitable regularity hypotheses on~$\eta$ (see \hyperlink{(D)}{{\rm(D)}} in Section \ref{s-MR} below), problem~\eqref{temperature}--\eqref{IC-temperature} has a unique solution~$(T,\uuu)$ in an appropriate functional class. Moreover, since the velocity field~$\uuu$ is the solution of the stationary Stokes system with Dirichlet's boundary condition, its evolution is uniquely determined by that of the temperature, so that we confine ourselves to the study of the large-time asymptotics of~$T$. 
 
The evolution of~$T$ takes place in the affine subspace~$\HH$ of the functions that belong to the Sobolev class~$H^1(D)$ and satisfy the boundary conditions in~\eqref{boundary-condition}. The independence of the random variables~$\eta_k$ entering~\eqref{eta-form} implies that the restriction of $T(t)$ to integer times form a homogeneous Markov process in~$\HH$, and our goal is to prove the existence, uniqueness, and exponential stability of the related stationary distribution. To this end, we impose the following four hypotheses\footnote{We give here a somewhat informal formulation of the hypotheses, referring the reader to Section~\ref{s-MR} and Condition~\hyperlink{(D)}{(D)} for more details.} on the law~$\ell$ of random variables~$\eta_k$:
\begin{description}
	\item [\underline{\rm Regularity}]
\sl The measure~$\ell$ is concentrated on the vector space~$\EE$ of functions $\zeta\in L^2([0,1],(H_0^1\cap H^2)(D))$ such that $\p_t \zeta \in L^2([0,1]\times D)$.
	\item [\underline{\rm Localisation}]
The support~$\KK\subset\EE$ of~$\ell$ is compact in the topology of~$\EE$, and for any $\zeta \in \KK$ and at any time~$t\in[0,1]$, $\zeta(t)$ is supported in the closure of a thin strip $D_c:=\{x\in D:0<x_3<c\}$.
\end{description}
In what follows, we denote by~$E$ the space of functions $\zeta\in\EE$ supported in the closure of~$(0,1)\times D_c$. Thus, by our assumptions, $\KK$ is a compact subset of~$E$. 

\begin{description}
	\item [\underline{\rm Non-degeneracy}]
\sl The measure~$\ell$ is non-degenerate in the following sense.  There is an orthonormal basis~$\{\varphi_j\}_{j\ge1}$ of~$E$ such that, for each~$j$, the projection~$\ell_j$ of~$\ell$ to the one-dimensional subspace spanned by~$\varphi_j$ possesses smooth density (with respect to Lebesgue measure) with the origin in its support.  Moreover, the measure~$\ell$ is the direct product of its projections~$\ell_j$, $j\ge1$.
	\item [\underline{\rm Amplitude}]
The support~$\KK$ contains a sufficiently large Hilbert cube\,\footnote{Recall that a Hilbert cube in~$E$ is a subset of the form $\{\zeta\in E:|(\zeta,e_j)|\le b_j\mbox{ for }j\ge1\}$, where $\{e_j\}$ is an orthonormal basis in~$E$, and $b_j>0$ are some numbers such that $\sum b_j^2<\infty$.} in~$E$. 
\end{description}

The following theorem is a simplified version of the main result of this paper. Its precise formulation can be found in Section~\ref{s-MR}.

\begin{mt}
	The discrete-time Markov process in~$\HH$ associated with problem~\eqref{temperature}, \eqref{velocity-field} has a unique stationary measure~$\mu$, which is concentrated on the Sobolev class~$H^2(D)$. Moreover, for any continuous observable $f:H^1(D)\to\R$ and any initial state $T_0\in\HH$, the ensemble average $\E f(T(k))$ converges  exponentially as $k \to \infty$ to the mean value $\langle f,\mu\rangle$. 
\end{mt}

The proof of this result is based on an abstract criterion for exponential mixing in infinite-dimensional spaces and a {\it new unique continuation property\/} for the linearisation of system~\eqref{temperature}, \eqref{velocity-field}. We refer the reader to Section~\ref{s-MR} for the general scheme used to establish the Main Theorem and to Sections~\ref{s-ExpoMixing} and~\ref{s-uniquecontinuation} for a detailed proof. 
 
Let us mention that the problem of mixing for Markov processes associated with PDEs subject to random perturbations has been the focus of attention of many researchers in the last thirty years. We refer the reader to the papers~\cite{FM-1995,KS-cmp2000,EMS-2001,BKL-2002} for the first results in this direction and to the review papers~\cite{ES-2000,bricmont-2002,flandoli-2008,debussche-2013} and the book~\cite{KS-book} for a detailed description of numerous results obtained in the case of a non-degenerate noise. The case when the random perturbation does not act directly on the determining modes of the dynamics is more complicated, and there are only a few results in that setting. In particular, Hairer and Mattingly~\cite{HM-2006,HM-2011} proved the exponential mixing for the 2D Navier--Stokes system on the torus and sphere, assuming that the external random force is white in time and acts on a few Fourier modes. The paper~\cite{FGRT-2015} established a similar result in the case of the Boussinesq system~\eqref{full-vel}, \eqref{full-temp}. A class of PDEs with a different type of noise acting on finitely many Fourier modes was studied in~\cite{KNS-gafa2020,KNS-jep2020}. The Navier--Stokes system with a noise localised in the physical space was investigated in~\cite{shirikyan-asens2015,shirikyan-jems2021}. Finally,  scalar conservation laws and order-preserving PDEs with random forcing were studied by Debussche--Vovelle~\cite{DV-2015} and Butkovsky--Scheutzow~\cite{BS-2020}. The present paper is close in spirit to~\cite{shirikyan-asens2015,KNS-gafa2020} and uses a mixing criterion established there. However, its application is not straightforward, and new ideas invoking the technique of Carleman estimates and unique continuation property are necessary; see the end of Section~\ref{s-MR} for further details and Section~\ref{s-carleman} for a full presentation.  

\smallskip
The paper is organised as follows. Section~\ref{s-EUR} establishes some preliminary results on the initial-boundary value problem for the Boussinesq system with infinite Prandtl number, the existence of a compact absorbing set, and the regularity of the resolving operator. The main result of the paper and the scheme of its proof are presented in Section~\ref{s-MR}. The details of the proof are given in Sections~\ref{s-ExpoMixing} and~\ref{s-uniquecontinuation}. In particular, Section~\ref{s-carleman} contains the main new technical tool, dealing with the unique continuation property for the linearised Boussinesq system. The appendix gathers some auxiliary results used in the main text.

\subsubsection*{Acknowledgments}
JF is partly supported by the National Science Foundation under the grant NSF-DMS-1816408. The research of AS was supported by the \textit{CY Initiative of Excellence\/} through the grant {\it Investissements d'Avenir\/} ANR-16-IDEX-0008 and by the Ministry of Science and Higher Education of the Russian Federation (megagrant agreement No.~075-15-2022-1115).

\subsubsection*{Notation and conventions}
We use  various spaces of functions defined in the strip
\begin{equation} \label{domain-D}
D=\{x=(x_1,x_2,x_3)\in\R^3: x_1,x_2\in\R/2\pi\Z,\,0<x_3<1\},	
\end{equation}
and in particular all functions are assumed to be $2\pi$-periodic in~$x_1$ and in~$x_2$.  We denote by~$L^p(D)$ and~$H^k(D)$ the Lebesgue and Sobolev spaces, respectively, and by~$H_0^1(D)$ the space of functions in~$H^1(D)$ with zero trace on the boundary of~$D$. We often write~$H^k$ and~$H^1_0$ instead of $H^k(D)$ and $H_0^1(D)$ respectively. We denote by $\langle\cdot,\cdot\rangle$ the $L^2$ inner product and by~$\|\cdot\|$ the corresponding norm.  

\smallskip
\noindent
Given an interval $J\subset\R$ and a separable Banach space~$X$ (or a Polish space~$X$ with a metric~$\dd_X$), we shall use the following classes of functions:
\begin{itemize}
\item
$L^p(J,X)$,  $1 \leq p < \infty$ is the space of Borel-measurable functions $f:J\to X$ such that
$$
\|f\|_{L^p(J,X)}:=\biggr(\int_J\|f(t)\|_X^p\dd t\biggr)^{1/p}<\infty,
$$
with a usual modification by supremum in the case $p=\infty$. 

\item
$H^k(J,X)$ is the space of the functions $f\in L^2(J,X)$ such that $\p_t^jf\in L^2(J,X)$ for any $1\le j\le k$.

\item
$C_b(X)$ is the space of bounded continuous functions $f:X\to\R$ endowed with the natural norm
$$
\|f\|_\infty=\sup_{u\in X}|f(u)|.
$$

\item
$L_b(X)$ is the space of functions $f\in C_b(X)$ that are Lipschitz continuous on~$X$; it is endowed with the norm
$$
\|f\|_L=\|f\|_\infty+\sup_{0\le \dd_X(u,v)\le 1}\frac{|f(u)-f(v)|}{\dd_X(u,v)}. 
$$

\item
We write~$\PP(X)$ for the set of  Borel probability measures on~$X$.  The set~$\PP(X)$ is considered with the topology of weak$^*$ convergence,  which is metrizable by the dual-Lipschitz distance:
$$
\|\mu_1-\mu_2\|_L^*
=\sup_{\|f\|_L\le 1} |\langle f,\mu_1\rangle-\langle f,\mu_2\rangle|,
$$
where the supremum is taken over all functions $f\in L_b(X)$ with norm not exceeding~$1$. 

\item 
Given a finite interval $J_\tau=[0,\tau]$, we define the underlying spaces for solutions
\begin{align*}
\XX_\tau&=L^2(J_\tau,H^2)\cap H^1(J_\tau,L^2),\\
\XX_\tau^0&=\{v\in\XX_\tau:v=0\text{ on }\p D\}, \\
	\UU_\tau&=L^2(J_\tau,H_0^1\cap H^4)\cap H^1(J_\tau,H_0^1\cap H^2),	
\end{align*}
where Sobolev's spaces in the definition of~$\XX_\tau$ and~$\UU_\tau$ denote respectively spaces of scalar and $\R^3$-valued functions.  If $\tau=1$, we  omit the subscript from the notation and write respectively $J$, $\XX$, $\XX^0$,  and $\UU$ instead of $J_1$, $\XX_1$, $\XX_1^0$,  and $\UU_1$.

\item Since $T$ attains non-trivial boundary conditions \eqref{boundary-condition}, given numbers $T_b,T_u\in\R$,  we denote by~$\HH\subset H^1(D)$ the affine subspace of functions satisfying the first two relations in~\eqref{boundary-condition}. 

\item Recall that the forcing acts only on a subdomain of $D$.  Hence,  given $c\in(0,1)$,  let $D_c=\{(x_1,x_2,x_3)\in D:0<x_3<c\}$.  We write $H_c^k(D)$ for the the space of functions belonging to $H^k(D)\cap H_0^1(D)$ that are supported in the closure of~$D_c$ and denote by~$E$ the space of functions $\zeta\in L^2(J,H_c^2(D))$ such that $\p_t \zeta\in L^2(J\times D)$.

\item
We denote by~$C$ unessential positive numbers that may vary from line to line. 
\end{itemize}

\section{Initial-boundary value problem and Markov process}
\label{s-EUR}
In this section, we discuss the existence, uniqueness, and regularity of solutions for system~\eqref{temperature}, \eqref{velocity-field} and describe an associated discrete-time Markov process. All the results of this section are rather standard, and therefore we do not present detailed proofs. 

\subsection{Existence and uniqueness of a solution}
\label{s-existence-uniqueness}
Let us denote by $M:L^2(D)\to (H_0^1\cap H^2)(D,\R^3)$ the bounded linear operator mapping a function~$T$ to the solution~$\uuu$ of the elliptic system~\eqref{velocity-field} with Dirichlet's boundary conditions (see~\eqref{boundary-condition}).  Note that $M$ is well defined due to the maximal regularity of stationary Stokes operator.  In addition,  $M$ is a bounded operator from $H^k(D)$ to $(H^1_0 \cap H^{k+2}) (D)$ for any $k \geq 0$ and we refer to this fact below as regularizing property of $M$.

Substituting $\uuu=M(T)$ into~\eqref{temperature}, we obtain the following non-local PDE:
\begin{equation} \label{boussinesq-reduced}
	\p_tT- \Delta T+\langle M(T),\nabla\rangle T=\eta(t,x). 
\end{equation}
Given $T_u,T_b\in \R$,  denote $\bT(x)=T_b+x_3(T_u-T_b)$,  where $x = (x_1, x_2,x_3)$. 

\begin{proposition}\label{p-EU}
For any $\tau \in (0, \infty)$, any function $T_0\in \HH$ and any (deterministic) right-hand side $\eta\in L^2(J_\tau,L^2)$,	problem~\eqref{boussinesq-reduced}, \eqref{IC-temperature} has a unique solution~$T\in \XX_\tau$ such that $T - \bT \in\XX_\tau^0$, and the corresponding vector field~$\uuu = M(T)$ belongs to the space~$\UU_\tau$.  
\end{proposition}

\begin{proof}
	We confine ourselves to the derivation of a priori estimates for solutions, since then the local existence and uniqueness of solution follow from a standard fixed point argument.  In addition,  a continuation of the local solution to the entire interval~$J_\tau$ is a standard consequence of a priori bounds. 

\smallskip
{\it Step~1: Reduction\/}. 	Define  $S = T - \bT$. Then $S = 0$ whenever $x_3 = 0$ or $x_3 = 1$, and the function $S_0 := S(x, 0) = T_0 - \bT$ belongs to $H^1_0(D)$.  In addition, $S$~satisfies the equation
\begin{equation} \label{fes}
	\p_t S -  \Delta S + \langle M(T),\nabla\rangle S + (T_u - T_b) M_3(T) = \eta(t,x), 
\end{equation}
where $M = (M_1, M_2, M_3)$. We claim that $M(T) = M(S)$. Indeed, by \eqref{velocity-field},  $\uuu = M(T)$ solves 
\begin{equation}
-\Delta \uuu +\nabla p = \Ra\, (S + \bT)\et, \quad \diver \uuu=0,
\end{equation}
supplemented with Dirichlet's boundary condition from \eqref{boundary-condition}.  Next,  note that the term $\Ra\, \bT= \Ra\, \partial_{x_3} (T_b x_3 + \frac{1}{2}(T_u-T_b)x_3^2)$ can be included into the pressure~$p$. Thus, $\uuu = M(T)$ solves Dirichlet's problem 
\begin{equation}\label{ues}
-\Delta \uuu +\nabla \tilde{p} = \Ra\,  S \et, \quad \diver \uuu=0\,.
\end{equation}
However, $M(S)$ satisfies~\eqref{ues} as well, with possibly modified pressure. The uniqueness of solutions of Stokes system implies $M(T) = M(S)$. 

Hence,  \eqref{fes} yields that the function~$S$ satisfies 
\begin{equation} \label{eqs}
	\p_t S - \Delta S + \langle M(S),\nabla\rangle S + (T_b - T_u) M_3(S) = \eta(t,x), 
\end{equation}
supplemented with the zero Dirichlet boundary conditions and compatible initial data. 

\smallskip
{\it Step~2: $L^2$-regularity\/}. 
Let us derive an a priori estimates for~$\uuu = M(S)$. Testing~\eqref{ues} with the divergence free~$\uuu$  yields
\begin{equation*}
\|\nabla \uuu\|^2 = \Ra\, \langle S,  u_3 \rangle 
\leq \Ra \,\|S\| \|\nabla \uuu\| 
\leq  \frac{1}{2}\|\nabla \uuu\|^2 + C \|S\|^2 \,,
\end{equation*}
where we used Poincar\'e's and Cauchy's inequalities. Note that~$C$ is a constant that depends only on~$\Ra$ and~$D$. Consequently, 
\begin{equation}\label{rou}
\|\nabla \uuu\| = \|M(S)\|_{H^1} \leq C \|S\| \,.
\end{equation}
Testing \eqref{eqs} with~$S$, we derive 
\begin{equation}\label{tse}
\frac{1}{2} \p_t \|S\|^2 +  \|\nabla S\|^2 = \langle \eta, S\rangle - (T_b - T_u) \langle M_3(S), S \rangle \,,
\end{equation}
where we used that $M(S)$ is divergence free. Using H\"older's and  Poincar\'e's inequalities and~\eqref{rou}, the terms on the right-hand side can be estimated as follows: 
\begin{align*}
| \langle \eta, S\rangle| &\leq \|\eta\| \,\|S\| \leq \frac{1}{2} \|\nabla S\|^2 + C  \|\eta\|^2 \,, \\
|(T_b - T_u)\langle M_3(S), S \rangle| &\leq C \|M(S)\| \,\|S\| \leq C \|S\|^2 \,.
\end{align*}
A substitution  into~\eqref{tse} results in
\begin{equation}\label{fel}
 \p_t \|S\|^2 +  \|\nabla S\|^2 \leq C \|S\|^2 + C \|\eta\|^2.
\end{equation} 
Gronwall's inequality implies that 
\begin{equation}\label{lis}
 \|S\|^2_{L^\infty( J_\tau, L^2)} 
+  
  \|S\|^2_{L^2( J_\tau, H^1)} 
 \leq C(\tau)( \|\eta\|_{L^2(J_\tau, L^2)}^2 + \|S_0\|^2)  \,, 
\end{equation}
where $C(\tau)$ depends on $\tau$ and fixed parameters of the problem. 

\smallskip
{\it Step~3: Higher regularity\/}.  Testing~\eqref{eqs} with~$\Delta S$, we derive 
\begin{equation}\label{mes}
\frac{1}{2} \p_t \|\nabla S\|^2 +  \|\Delta S\|^2 = - \langle \eta, \Delta S\rangle + (T_b - T_u)\langle M_3(S), \Delta S \rangle +  \bigl\langle \langle M(S),\nabla\rangle S, \Delta S \bigr\rangle \,.
\end{equation}
As above, we estimate the first two terms on the right-hand side: 
\begin{align*}
| \langle \eta, \Delta S\rangle| &\leq \|\eta\| \|\Delta S\| \leq \frac{1}{6} \|\Delta S\|^2 + C   \|\eta\|^2 \,, \\
|(T_b - T_u)\langle M_3(S), \Delta S \rangle| &\leq C \|M(S)\| \|\Delta S\| \leq  \frac{1}{6} \|\Delta S\|^2 + C   \|S\|^2 \,.
\end{align*}
To bound the last term, note that, by elliptic regularity, we have $\|S\|_{H^2} \leq C \|\Delta S\|$. Using H\" older's inequality, Sobolev's embedding, \eqref{rou}, and interpolation, we derive
\begin{align*}
\bigl|\bigl\langle \langle M(S),\nabla\rangle S, \Delta S \bigr\rangle\bigr| 
&\leq \|\Delta S\| \|M(S)\|_{L^6} \|\nabla S\|_{L^3}\leq 
\|\Delta S\|  \|M(S)\|_{H^1} \|S\|_{H^{\frac{3}{2}}}  \\
&\leq C \|\Delta S\|  \|S\| \|S\|_{H^2}^{\frac{3}{4}} \|S\|^{\frac{1}{4}}
\leq \frac{1}{6}  \|\Delta S\|^2 + C \|S\|^{10} \,.
\end{align*}
A substitution  into~\eqref{mes} yields
\begin{equation} \label{diffin-forS}
 \p_t \|\nabla S\|^2 +  \|\Delta S\|^2 \leq C  (1 +  \|\eta\|^2 +  \|S\|^{10}) \,.
\end{equation}
Consequently, \eqref{lis} implies that 
\begin{equation}\label{sht}
 \|S\|^2_{L^\infty(J_\tau, H^1)}+\|S\|^2_{L^2(J_\tau, H^2)} 
 \leq C(\tau)( \|\eta\|_{L^2(J_\tau, L^2)} + 1 +  \|S_0\|)^{10} + \|S_0\|^{2}_{H^1}\,.
\end{equation}
Furthermore, it follows from~\eqref{eqs} that $S \in H^1(J_\tau, L^2)$. The same assertions clearly hold true for~$T$, and therefore $T \in \XX_\tau$. Finally, the elliptic regularity of Stokes' operator implies that $\uuu \in \UU_\tau$.
\end{proof}

\subsection{Compact absorbing set}
\label{s-absorbing-set}

In this section,  we show that there exists a bounded set in~$H^3$ to which any solution of~\eqref{boussinesq-reduced} belong for large time. This implies, in particular, that there is a compact absorbing set in~$H^1$. The $H^3$ regularity of solutions will also be needed in the proof of unique continuation for the linearised problem. Our assertions follow from  a priori estimates, which are based on standard multipliers techniques. Since the results hold path-wise,  we prove them for a fixed realisation of the noise,  that is,  for deterministic~$\eta$ with the specified regularity. 

\begin{proposition}\label{p-dissipativity}
Suppose there are $r, r_0 > 0$ such that the right-hand side in~\eqref{boussinesq-reduced} has the form~\eqref{eta-form} where $\{\eta_k\}_{k\ge1}$ belongs to a fixed ball~$B_{\XX^0}(r)$,  and the boundary temperatures~$T_u$ and~$T_b$ satisfy the inequality $|T_u|+|T_b|\le r_0$.  Then,  there is $C=C(r, r_0)>0$ such that,  for any $T_0\in \HH$,  with $\|T_0\|_{H^1}\le R$,  there is an integer~$\tau=\tau(r, r_0, R)\ge1$
with 
\begin{equation} \label{tau-bound}
	\tau\le C(r, r_0)\ln(R+2),
\end{equation}
such that the solution~$T$ of~\eqref{boussinesq-reduced}  satisfies the inequality 
\begin{equation} \label{H2norm}
	\|T(t)\|_{H^3}\le C(r, r_0)\quad\mbox{for any $t\ge\tau$}. 
\end{equation}
\end{proposition}

\begin{proof}
{\it Step~1. Reduction to $H^2$ initial condition\/}. Since $\eta_k \in B_{\XX^0}(r)$ for all $k\ge1$, Theorem~3.1 in~\cite[Chapter~I]{LM1972} implies that
\begin{equation} \label{H1-eta}
	\rho:=\sup_{t\ge 0}\|\eta(t)\|_{H^1}\le C_1r. 
\end{equation}
By \eqref{sht},  there is $t_0\in(0,1)$ such that
$$
\|T(t_0)\|_{H^2}
\le \|\bar T\|_{H^2}+\|S(t_0)\|_{H^2}
\le C_2(1 + \rho+R+r_0)^{5}. 
$$
We shift our initial time to $t_0$ and assume without loss of generality that 
\begin{equation} \label{T0H2}
\|T_0\|_{H^2} \leq C_2(1 + \rho + R + r_0)^{5} =: R_1. 
\end{equation}
In particular, by Sobolev's embedding, we have $T_0\in C(\overline D)$ and 
\begin{equation} \label{Linfty-norm}
	\|T_0\|_{L^\infty}\le C_3R_1. 
\end{equation}

\smallskip
{\it Step~2. Bound for the $H^1$-norm\/}. 
Let $T^*$ be the solution of the linear parabolic problem
\begin{align} \label{tte}
	\p_t T^*- \Delta T^*+ \langle M(T), \nabla\rangle T^*&=\eta(t,x), \\
	T^*(0)= T_0, \quad 
		T^*\bigr|_{\partial D} &= 0 \,.\label{T*IB}
\end{align}
Since~$T$ solves~\eqref{boussinesq-reduced},  the function $T^\dagger := T - T^*$ satisfies the equations 
\begin{equation}\label{etd}
\begin{aligned}
	\p_t T^\dagger - \Delta T^\dagger+ \langle M(T), \nabla \rangle T^\dagger &= 0, \\
	T^\dagger(0) = 0, \quad T^\dagger\bigr|_{x_3=0}=T_b, \quad
	T^\dagger\bigr|_{x_3=1}&=T_u.
\end{aligned}
\end{equation}
We first estimate~$T^*$.  Testing~\eqref{tte} with $T^*$,  integrating by parts,  using that $M(T)$ is divergence free,  and usign Poincar\' e's and Young's inequalities,  we obtain
\begin{equation*}
\frac{1}{2}\p_t \|T^*\|^2+\|\nabla T^*\|^2
= \langle \eta, T^*\rangle 
\leq \|\eta\| \,\|T^*\| 
\leq \frac{1}{2} \|\nabla T^*\|^2 + C \|\eta\|^2 \,.
\end{equation*}
A rearrangement yields 
\begin{equation*}
\p_t  \|T^*\|^2 + \|\nabla T^*\|^2 \leq 2C  \|\eta\|^2.
\end{equation*}
Applying Poincar\'e's and Gronwall's inequalities and using~\eqref{H1-eta}, we derive
\begin{equation} \label{L2normT*}
\|T^*(t)\|^2 
\leq e^{- \lambda_1 t} \|T_0\|^2 + (C_4r)^2,
\end{equation}
where $\lambda_1>0$ is the first eigenvalue of the Dirichlet Laplacian in~$D$. 

To estimate $T^\dagger$, we note that $M(T)$ is a continuous function of time with range in~$L^\infty(D)$ and apply the maximum principle to~\eqref{etd}. Since $|T_u|+|T_b|\le r_0$, we obtain
\begin{equation} \label{Linfty-Tdagger}
\|T^\dagger(t)\|_{L^\infty}\le r_0 \quad\mbox{for all $t \ge0$}. 
\end{equation}
Combining~\eqref{L2normT*} and \eqref{Linfty-Tdagger}, we derive
$$
\|T(t)\|\leq e^{-\lambda_1t/2}\,\|T_0\|+C_4r+r_0|D|\quad\mbox{for $t\ge0$}. 
$$
In view of~\eqref{Linfty-norm}, we can find a universal constant $C_5>0$ such that
\begin{equation} \label{ltt}
\|T(t)\|\le C_5(r+r_0+1)\quad\mbox{for $t\ge \tau_1(R_1):=C_5\ln R_1$}. 
\end{equation}

We now fix any $s \in [0, 1]$ and any $t\ge \tau_1+1$, and integrate~\eqref{diffin-forS} over an interval $[t - s, t + 1]$ of length~$1 + s$.  Using the relation $T=\bar T+S$ (recall $\bar T = T_b + x_3(T_u - T_b)$) and inequality~\eqref{ltt}, for $t\ge \tau_1+1$ we derive
\begin{equation} \label{H1-norm-p}
\|T(t+1)\|_{H^1}^2+\int_t^{t+1}\|T\|_{H^2}^2\dd\tau
\le C_6(r + r_0 + 1)^{10} + \|S(t - s)\|_{H^1}^2 \,,
\end{equation}
where $C_6$ is independent of $s$ and $t$. Integrating~\eqref{H1-norm-p} in~$s\in[0,1]$ and using that all terms except the last one are independent of~$s$, we obtain 
\begin{equation*} 
\|T(t+1)\|_{H^1}^2+\int_t^{t+1}\|T\|_{H^2}^2\dd\tau\le C_6(r + r_0 + 1)^{10} + \int_{t - 1}^t \|S\|_{H^1}^2\dd\tau\,.
\end{equation*}
Integrating~\eqref{fel}, using~\eqref{ltt}, and taking the supremum over $t\ge\tau_1+1$, we arrive at the inequality
\begin{equation} \label{H1-norm}
	\sup_{t\ge \tau_1+2}\biggl(\|T(t)\|_{H^1}^2+\int_t^{t+1}\|T\|_{H^2}^2\dd\tau\biggr)\le C_7(r + r_0 + 1)^{10}\,.
\end{equation}

\smallskip
{\it Step~3. Bound for the $H^2$-norm\/}. 
Fix any $t \geq \tau_1+3$.  Then,  by \eqref{H1-norm} there is $t_0 \in [t-1,t]$ such that 
\begin{equation}
\|T (t_0) \|_{H^2} \leq C(r + r_0 + 1)^{5}, 
\end{equation}
and the same estimate holds for $S(t_0)$,  where $S = T - \bar{T}$ (see Proposition \ref{p-EU}).  In addition,  $S$ satisfies \eqref{H1-norm} with~$T$ replaced by~$S$ and solves the equation
\begin{equation} \label{heat-f}
	\p_t S-\Delta S = f(t,x):=\eta(t,x)-\langle M(T),\nabla\rangle S-(T_u-T_b)M_3(T),
\end{equation}
supplemented with zero boundary conditions.  We claim that $f \in L^2([t_0, t], H^1)$, with a bound that depends only on~$r$ and~$r_0$.  Indeed,  by assumption, we have $\eta \in L^2([t_0, t], H^1)$. Using the smoothing properties of~$M$ and inequality~\eqref{ltt}, for any~$s \geq t_0$ we derive
\begin{equation}\label{MT-H2}
\|M(T(s))\|_{H^2} \leq C \|T(s)\|_{L^2} \leq C(r + r_0 + 1) \,.
\end{equation}
Combining this with the embedding $H^2 \hookrightarrow L^\infty$ and inequality~\eqref{H1-norm}, for any $s\geq t_0$ we obtain 
\begin{align*}
\|\langle M(T),\nabla\rangle S\|_{H^1} 
&\leq C (\|\nabla M(T)\|_{L^\infty} \|S\|_{H^1} + \|M(T)\|_{L^\infty}\|S\|_{H^2})\\
&\leq C (\|M(T)\|_{H^3} \|S\|_{H^1} + \|M(T)\|_{H^2} \|S\|_{H^2}) \\
&\leq C (\|T\|_{H^1} \|S\|_{H^1} + \|T\|\, \|S\|_{H^2}) \\
&\leq C \bigl((r + r_0 + 1)^{10} +  (r + r_0 + 1) \|S\|_{H^2}\bigr)\,,
\end{align*}
where we skipped the argument~$s$ of all the functions to simplify formulas. Hence,  using~\eqref{H1-norm} and~\eqref{ltt}, we can write
\begin{equation}\label{f-L2H1}
\|f\|_{L^2([t_0, t], H^1)} \leq C (r + r_0 + 1)^{10} \,.
\end{equation}
Finally,  the parabolic regularity of the heat equation yields 
\begin{equation}\label{H2-bound}
\begin{aligned}
\|S\|_{L^\infty([t_0, t], H^2)} + \|S\|_{L^2([t_0, t], H^3)} 
&\leq C\bigl(\|S(t_0)\|_{ H^1} + \|f\|_{L^2([t_0, t], H^1)}\bigr) \\
&\leq  C (r + r_0 + 1)^{10}.
\end{aligned}
\end{equation}
Since $\bar{T}$ is smooth,  the same bound holds with~$S$ replaced by~$T$.

\smallskip
{\it Step~4. Bound for the $H^3$-norm\/}. 
With \eqref{H2-bound} at our disposal,  we reiterate the procedure from Step 3 and obtain an improved regularity. Fix any $t \geq \tau_1+3$.  Then,  by~\eqref{H2-bound}, there is $t_0 \in [t - 1, t]$ such that 
\begin{equation}
\|S (t_0) \|_{H^3} \leq C(r + r_0 + 1)^{10}. 
\end{equation}
In view of~\eqref{heat-f} and the boundedness of~$\bar T$ in~$H^3$, inequality~\eqref{H2norm} will be established if we prove that the function~$f$ defined in Step 3 belongs to the space $L^2([t_0, t], H^2)$, with a bound that depends only on~$r$ and~$r_0$; cf.~\eqref{f-L2H1}. 

To see this, we note that $\eta \in L^2([t_0, t], H^2)$ by assumption. Since the space~$H^2$ is a ring,  using~\eqref{MT-H2}, for $s\in[t_0,t]$ we can write
\begin{align*}
\|M(T(s)),\nabla\rangle S(s)\|_{H^2}
&\le C\,\|M(T(s))\|_{H^2}\,\|\nabla S(s)\|_{H^2}\\
&\le C\,(r + r_0 + 1)\,\|S(s)\|_{H^3}. 
\end{align*}
Combining this with~\eqref{MT-H2} and~\eqref{H2-bound}, we thus obtain the bound
\begin{equation}
\|f\|_{L^2([t_0, t], H^2)} 
\leq C (r + r_0 + 1)^{11}\,.
\end{equation}
This completes the proof of the proposition. 
\end{proof}

\subsection{Regularity of the resolving operator}
\label{s-resolving}
In view of Proposition~\ref{p-EU}, we can define a function~$\sS:\HH\times L^2(J\times D)\to H^1$ mapping a pair~$(T_0,\eta)$ to~$T(1,\cdot)$, where $T\in\XX$ is the solution of problem~\eqref{boussinesq-reduced}, \eqref{IC-temperature}. The following proposition establishes a regularity property for~$\sS$. 

\begin{proposition}\label{p-resolving-operator}
The mapping $\sS:\HH\times L^2(J\times D)\to H^1(D)$ is analytic. 
\end{proposition}

\begin{proof}
As in the proof of Proposition~\ref{p-EU}, we set $\bar T=T_b+x_3(T_u-T_b)$ and  write $T=\bar T+S$. The function~$S$ belongs to~$\XX^0$,  solves~\eqref{eqs}, with zero Dirichlet boundary condition, and satisfies $S(0)=S_0 := T_0 - \bT \in H^1_0(D)$. Notice that, to prove the analyticity of~$\sS$, it suffices to show that $S(1)$ depends  analytically on $(S_0,\eta)\in H_0^1\times L^2(J\times D)$. 
 
For given $S_0\in H_0^1$ and $\eta\in L^2(J\times D)$, let $z = A(S_0,\eta)$ be the solution of the inhomogeneous heat equation 
\begin{equation}
\p_t z -  \Delta z = \eta, \qquad z(x, 0) = S_0(x)
\end{equation}
with zero Dirichlet boundary conditions. The standard parabolic regularity yields $z \in \XX^0 \cap C(J_\tau, H^1)$, so that~$A$ is a linear and continuous (hence, analytic) operator of its arguments. In the case when $S_0=0$, we write $A\eta$  instead of~$A(0,\eta)$.

A solution of~\eqref{eqs} is sought in the form $S=Av+z$, where $z=A(S_0,\eta)$ and $v\in\ZZ:=L^2(J\times D)$.  Since $S$ solves \eqref{lis},  
then $w = Av$ satisfies 
\begin{equation}\label{eqaw}
w_t - \Delta w + G(w,  z) = 0 \,, \qquad w(x, 0) = 0 \,,
\end{equation}
where for $w,z\in\XX^0$ we set 
\begin{equation} \label{Gwz}
G(w,z)=\langle M(w + z),\nabla\rangle (w + z)+(T_b - T_u)M_3(w + z).	
\end{equation}
In a different notation,  $v$ satisfies  
\begin{gather}\label{wwe}
F(v,z):=v+G(Av,z)=0 \,.
\end{gather}
In addition, by Proposition~\ref{p-EU}, for any $z\in\XX^0$ there is a unique $w \in \XX^0$,  satisfying \eqref{eqaw},  and therefore unique $v\in\ZZ$ satisfying~\eqref{wwe}. 
The regularizing properties of the linear operators~$M$ and~$A$ yield that $F$ is an analytic (in fact polynomial)  map from $\ZZ \times \XX^0$ to $\ZZ$.

Next, we claim that for any fixed $(v,z)\in\ZZ\times\XX^0$ satisfying $F(v, z) = 0$, the derivative $D_v F(v, z)$ is an invertible linear operator on~$\ZZ$. Indeed,  by Banach's inverse mapping theorem, it suffices to show that $D_v F(v, z)$ is a bijection from~$\ZZ$ to~$\ZZ$. Thus, it is enough  to prove that, for each $f \in \ZZ$, the equation   
\begin{equation} \label{linearised}
	D_v F(v, z)[\varphi] = f
\end{equation}
has a unique solution $\varphi\in\ZZ$.  In view  of~\eqref{wwe} and~\eqref{Gwz}, we have
$$
D_v F(v, z)[\varphi]=\varphi + \langle M( A\varphi),\nabla\rangle(w + z) + \langle M(w + z),\nabla\rangle (A\varphi) + (T_b - T_u)M_3(A\varphi),
$$
where $w=Av$.  By the parabolic regularity,  $A: \ZZ \to \XX^0$ is an invertible map,  and consequently if 
$\psi := A \varphi$,  then the unique solvability of~\eqref{linearised} is equivalent to the existence of unique solution $\psi \in \XX^0$ of the equation
\begin{gather*}\label{ppf}
 \p_t \psi -  \Delta \psi + \langle M( \psi),\nabla\rangle(w + z) + \langle M(w + z),\nabla\rangle \psi + (T_b - T_u)M_3(\psi) = f,
\end{gather*}
supplemented with the zero initial and boundary conditions. 
This problem is the linearisation of~\eqref{boussinesq-reduced}, \eqref{IC-temperature}, and its unique solvability in~$\XX^0$ follows by standard energy estimates,  combined with
regularization of $M$ and Gronwall's inequality (cf.  proof of Proposition~\ref{p-EU}.)

Finally, since~$F$ is an analytic function with invertible partial derivative $D_v F(v, z)$, by the analytic implicit function theorem (see~\cite[Theorem~3.3.2]{berger1977} or~\cite{dieudonne1969}), the unique solution~$v\in\ZZ$ depends analytically on~$z\in\XX^0$.  Also,  $z$ is a solution of heat equation,  which depends analytically on the data $(S_0,\eta)\in H_0^1\times\ZZ$.  By 
\eqref{sht},  the map $S \in \XX^0 \mapsto S(1) \in H^1_0$ is a well defined linear operator,  and therefore it is analytic.  Overall,  we proved that   
$S(1)$ depends analytically on $(S_0,\eta)\in H_0^1\times\ZZ$, as desired.
\end{proof}

\begin{corollary} \label{c-analytic}
	The map~$\sS$ is analytic from $\HH\times \XX$ to $H^2(D)$.
\end{corollary}

\begin{proof}
As was established in the proof of Proposition~\ref{p-resolving-operator}, the solution of~\eqref{boussinesq-reduced}, \eqref{boundary-condition}, \eqref{IC-temperature} can be written in the form 
\begin{equation} \label{representation-T}
T=\bar T+A(T_0-\bar T,\eta)+Av	
\end{equation}
where~$v\in L^2(J\times D)$ is the unique solution of~\eqref{wwe}.  Parabolic regularity combined with the same argument as
in the proof of Proposition~\ref{p-resolving-operator}  above shows that if $\eta \in \XX$,  then~\eqref{wwe} has a unique solution~$v\in\XX^0$,  which depends analytically on~$(T_0,\eta)\in \HH\times\XX$. 
It remains to note that the parabolic regularity yields that the restriction of~$Av$ to $t=1$ is a continuous linear operator from $H_0^1\times\XX$ to $H^2(D)$, and the required result follows immediately from~\eqref{representation-T}.
\end{proof}

\section{Main result}
\label{s-MR}

For fixed $c$ recall the definition of $H_c^k(D)$,  $E$, and $\HH$ in the notation section. 
In what follows, we often omit~$D$ from notation and write~$H^s$ for the Sobolev spaces on~$D$. 
Let us formulate the hypotheses imposed on the random perturbation~$\eta$.  We assume that $\eta$ has the form~\eqref{eta-form},  where~$\{\eta_k\}$ are i.i.d.\ random variables in~$E$ that satisfy the following hypothesis. 

\begin{itemize}
	\item [\hypertarget{(D)}{\bf(D)}]
	{\sl There is an orthonormal basis $\{\varphi_j\}_{j\ge1}$ in~$E$ such that 
	\begin{equation} \label{eta_k}
		\eta_k(t,x)=a\sum_{j=1}^\infty b_j\xi_j^k \varphi_j(t,x), 
	\end{equation}
	where $a>0$ is a parameter, $\{b_j\}$ are positive numbers satisfying the condition $\sum_jb_j^2<\infty$, and~$\{\xi_j^k\}_{j, k \geq 1}$ are independent random variables such that $|\xi_j^k|\le 1$ almost surely. Moreover, for any $j, k\ge1$, the law of~$\xi_j^k$ has a  density~$\rho_j\in C^1(\R)$ such that $\rho_j(0)>0$.}
\end{itemize}
Under the above regularity conditions on~$\eta_k$,  equations~\eqref{temperature}, \eqref{velocity-field} generate a random flow in~$\HH$. Moreover, the restrictions of the temperature function $T$ 
to integer times satisfy the relation
\begin{equation} \label{RDS}
	T_k=\sS(T_{k-1},\eta_k), \quad k\ge1,
\end{equation}
where $T_k=T(k,x)$, and~$\sS:\HH\times E\to \HH$ stands for the time-$1$ shift along trajectories of~\eqref{boussinesq-reduced}, with an initial condition specified at $t=0$. Since~$\{\eta_k\}$ are i.i.d.\ random variables, the family of all trajectories of~\eqref{RDS} form a Markov process in~$\HH$, which will be denoted by~$(T_k,\IP_T)$. Here,~$T_k$ stands for the trajectory of the process and~$\IP_T$ for the probability measure corresponding to the initial state $T\in \HH$. We denote by~$P_k(T,\Gamma)$ the transition function for~$(T_k,\IP_T)$ and write~$\PPPP_k$ and~$\PPPP_k^*$ for the corresponding Markov semigroups acting in the spaces~$C_b(\HH)$ and~$\PP(\HH)$, respectively. The following theorem describes the large-time asymptotics of trajectories for~\eqref{RDS} and is the main result of this paper.

\begin{theorem} \label{MT-mixing}
	There is a number $a^*>0$ such that, if 
	 Hypothesis~\hyperlink{(D)}{{\rm(D)}} with some $a\ge a^*$ is satisfied for the i.i.d.\ random variables~$\eta_k$ entering~\eqref{eta-form}, then the following assertions hold.
\begin{description}
	\item[\bf Uniqueness.] The Markov process~$(T_k,\IP_T)$ associated with~\eqref{temperature}, \eqref{velocity-field} has a unique stationary distribution $\mu \in\PP(H^2\cap\HH)$.
	\item[\bf Exponential mixing.]

	There are positive numbers $C$ and~$\gamma$ such that, for any $\lambda\in\PP(\HH)$, we have
	\begin{equation} \label{exponential-mixing}
		\|\PPPP_k^*\lambda-\mu\|_L^*
		\le C e^{-\gamma k}\bigl(1+\mmmm_1(\lambda)\bigr), \quad k\ge0,
	\end{equation}
where 
$$
\mmmm_1(\lambda) := \int_{\HH} \|u\|_{H^1} \lambda (\dd u)  
$$ 
is the mean value of the $H^1$ norm with respect to~$\lambda$, and~$\|\cdot\|_L^*$ denotes the dual-Lipschitz metric over the space~$H^1$. 
\end{description}
\end{theorem}

Let us note that the velocity field~$\uuu$ also converges to a limiting measure, since it can be reconstructed from the temperature by the relation~$\uuu=M(T)$, and the corresponding map in the space of measures is Lipschitz-continuous in the metric $\|\cdot\|_L^*$. Furthermore, using the methods described in~\cite[Sections~3.1 and~3.2]{KS-book},  one can prove that~$\PPPP_k^*$ is a contraction in the space of measures with a finite moment, provided that $k\ge1$ is sufficiently large.  Since the proof of this fact does not involve any new ideas, we do not discuss the details. 

\smallskip
Let us describe the main idea of the proof of Theorem~\ref{MT-mixing}, postponing the details to the next two sections. We first remark that the discrete-time random dynamical system associated with~\eqref{boussinesq-reduced} possesses a compact absorbing set $X\subset \HH$ such that, for any $R\ge1$ and $T\in B_{H^1}(R)\cap\HH$, we have $\IP_T\{T_k\in X\}=1$ for $k\gtrsim \ln R$. This implies that it suffices to consider the restriction of the dynamics to~$X$.

According to Theorem~\ref{t-expomixing},  the uniqueness and exponential mixing in  Theorem~\ref{MT-mixing} follows  if we prove the following properties:
\begin{itemize}
	\item [\hypertarget{(ma)}{\bf(a)}]
	{\sl regularity and smoothing for the resolving operator of~\eqref{boussinesq-reduced}};
	\item [\hypertarget{(mb)}{\bf(b)}]
	{\sl approximate controllability to a point for the nonlinear problem and global approximate controllability for the linearised equation};
	\item [\hypertarget{(mc)}{\bf(c)}]
	{\sl structural property for the noise}. 
\end{itemize}
Property~\hyperlink{(ma)}{(a)} is a standard result in the theory of parabolic-type PDEs (see Section~\ref{s-resolving}), and~\hyperlink{(mc)}{(c)} is satisfied due to the hypotheses imposed on the noise (see~\eqref{eta_k}). Verification of~\hyperlink{(mb}{(b)} is more involved and contains the main difficulties. We prove below that there is a temperature profile~$\bar T(x_3)$ that is accessible from any point $T_0\in X$, provided that the parameter~$a\ge1$ in~\eqref{eta_k} is sufficiently large (see Section~\ref{s-global-control}). Furthermore, using Carleman-type estimates due to Fabre--Lebeau~\cite{FL-1996}, we establish a unique continuation property for the dual of the linearised problem, which implies the required approximate controllability (see Section~\ref{s-uniquecontinuation}).

\section{Proof of exponential mixing}
\label{s-ExpoMixing}

This section contains the proof of Theorem \ref{MT-mixing}.  First,  in Section~\ref{s-reduction}, we show that Theorem~\ref{MT-mixing} follows once we check the assumptions of Theorem~\ref{t-expomixing}.  Conditions~\hyperlink{(H1)}{(H$_1$)} and~\hyperlink{(H4)}{(H$_4$)} are easy to verify, and the validity of~\hyperlink{(H2)}{(H$_2$)} is shown in Section~\ref{s-global-control}.  In Section~\ref{s-density-derivative}, we prove that~\hyperlink{(H3)}{(H$_3$)} is equivalent to Proposition \ref{p-UC}, which is the main technical novelty of this paper and is established in Section~\ref{s-uniquecontinuation}. 

\subsection{Reduction to Theorem~\ref{t-expomixing}}
\label{s-reduction}
In this section, we assume that the hypotheses of Theorem~\ref{t-expomixing} are satisfied for the discrete-time Markov process~$(T_k,\IP_T)$ associated wit~\eqref{boussinesq-reduced} and prove Theorem~\ref{MT-mixing}. By  Hypothesis~\hyperlink{(D)}{{\rm(D)}} and orthonormality of the basis, the support~$\KK_a$ of the law of~$\eta_1$ is a compact subset of~$\XX^0$. In particular, for any $a>0$, there is $r(a)>0$ such that $\eta_k\in B_{\XX^0}(r(a))$ almost surely for all~$k\ge1$. Let $C_a:=C(r(a),r_0)$ be the number defined in Proposition~\ref{p-dissipativity} and, given $R>0$, let $N_a(R)$ be the least integer greater than $C_a\ln(R+2)$. Then 
\begin{equation}\label{absorbing-ball}
\IP_T\bigl\{T_k\in B_{H^3}(C_a)\mbox{ for }k\ge N_a(R)\bigr\}=1\quad\mbox{for any $T\in B_{H^1}(R)\cap\HH$}.
\end{equation}
Setting $k_a=N_a(C_a)$, we define 
\begin{equation} \label{set-Xa}
X_a=\bigcup_{l=0}^{k_a-1}\sS^l \bigl(B_{H^3\cap\HH}(C_a);\KK_a,\dots,\KK_a\bigr), 
\end{equation}
where the map $\sS^l:\HH\times E^l \to\HH$,  $l \geq 1$ is defined recursively by the rules
\begin{equation}\label{dfsl}
\sS^1=\sS, \quad \sS^l(T;\zeta_1,\dots,\zeta_l)=\sS^{l-1}\bigl(\sS(u,\zeta_1); \zeta_2,\dots,\zeta_l\bigr), \quad l\ge2.
\end{equation}
It follows from relation~\eqref{absorbing-ball} with $R=C_a$ that~$X_a$ is invariant and absorbing for the Markov process~$(T_k,\IP_T)$ in the sense that 
\begin{align}
	\IP_T\{T_k\in X_a\mbox{ for all $k\ge0$}\}&=1
	\quad\mbox{for any $T\in X_a$}, \label{invariance}\\
	\IP_T\bigl\{T_k\in X_a\mbox{ for all $k\ge N_a(R)$}\bigr\}&=1
	\quad\mbox{for any $T\in B_{H^1}(R)\cap\HH$}. \label{absorption}
\end{align}
Furthermore, the continuity properties of~$\sS$ and the compactness of the inclusions $\KK_a\subset\XX^0$ and $B_{H^3\cap\HH}(C_a)\subset H^2$ imply that~$X_a$  is a compact subset in the space $H^2\cap\HH$. Applying the Bogolyu\-bov--Krylov argument (e.g.  see~\cite[Section~3.1]{DZ1996}), we see that the Markov process~$(T_k,\IP_T)$ possesses at least one stationary measure~$\mu\in\PP(\HH)$. Moreover, it follows from~\eqref{absorption} that any stationary measure is supported in~$X_a$. 

\smallskip
We next show that if, for some $a>0$, Theorem~\ref{t-expomixing} is applicable to the restriction of the Markov process~$(T_k,\IP_T)$ to $X_a$, then~\eqref{exponential-mixing} holds, and hence~$\mu$ is the unique stationary measure. The definition of the Markov semigroup shows that, to prove~\eqref{exponential-mixing},  it suffices to establish the inequality 
\begin{equation} \label{expo-decay-bound}
	\|P_k(T,\cdot)-\mu\|_L^*\le Ce^{-\gamma k}(1+\|T\|_{H^1})
\end{equation}
for any $T\in \HH$ and $k\ge0$. Let $N_a=N_a(\|T\|)$ be the integer in~\eqref{absorption}. Then the measure $\lambda:=P_{N_a}(T,\cdot)$ is supported in~$X_a$, and~\eqref{mixing-dL} applies. Hence,  if $k\ge N_a$, then the Kolmogorov--Chapman relation implies that 
$$
\|P_k(T,\cdot)-\mu\|_L^*=\|\PPPP_{k-N_a}^*\lambda-\mu\|_L^*\le Ce^{-\gamma (k-N_a)}\le C_1e^{-\gamma k}(1+\|T\|_{H^1}),
$$
where we used the explicit form of~$N_a$ and assumed  that $\gamma = \gamma(a) > 0$ is so small that $\gamma C_a\le 1$. If $0\le k<N_a$,  then by the definition of the dual-Lipschitz norm we have 
\begin{equation*}
\|P_k(T,\cdot)-\mu\|_L^*\le 2 
\leq 2e^{-\gamma k} e^{\gamma N_a} 
\leq Ce^{-\gamma k}(1+\|T\|_{H^1}) \,,
\end{equation*}
where we again used the inequality $\gamma C_a\le 1$. 

\smallskip
Thus, to complete the proof of Theorem~\ref{MT-mixing}, it suffices to show that, for any sufficiently large~$a>0$, the hypotheses of Theorem~\ref{t-expomixing} are satisfied for the restriction of~$(T_k,\IP_T)$ to~$X_a$. Hypothesis~\hyperlink{(H1)}{(H$_1$)}, in which $H = H^1$,  $\HH$ is defined in Section~\ref{s0}, and $V=H^2(D)$, is a consequence of Corollary~\ref{c-analytic}. Hypothesis~\hyperlink{(H4)}{(H$_4$)}  is postulated in Condition~\hyperlink{(D)}{(D)} imposed on the noise~$\eta_k$.  Verification of \hyperlink{(H2)}{(H$_2$)} and~\hyperlink{(H3)}{(H$_3$)} is done in the next two subsections.

\subsection{Global approximate controllability to a point}
\label{s-global-control}
In this subsection,  we show that the problem is globally controllable to a single point, which is independent of the starting position.  Specifically,  we show that Hypothesis~\hyperlink{(H2)}{(H$_2$)} of Theorem~\ref{t-expomixing} holds true for $\sS^k$ defined in \eqref{dfsl},  and $\sS$ as in Theorem \ref{MT-mixing}.   We remark that problem \eqref{temperature}--\eqref{IC-temperature} with $\eta \equiv 0$ has complicated dynamics if~$\Ra$ is large and, in particular,  it does not have a globally stable equilibrium.  This suggests that if~$\eta$ is small, then the system might not be controllable. It is exactly at this point where we need a large parameter~$a>0$ in front of the sum in~\eqref{eta_k}. However, one needs to be careful, since the initial conditions~$T_0$ belongs to the set~$X_a$, which depends on~$a$. To avoid any circular reasoning, we first choose $\eta = 0$ and drive the solution to a set which is independent of~$a$. Then, we control the trajectory with an appropriately chosen~$\eta$. The proof is divided into three steps. 

\smallskip
{\it Step~1: Reduction to bounded initial conditions\/}.
Fix any $a > 0$ and let $R>0$ be so large that $X_a\subset B_{H^1}(R)$,  where $X_a$ was defined by \eqref{invariance} and \eqref{absorption}. 
Let~$C$ and~$\tau$ be as in Proposition~\ref{p-dissipativity} with $r=0$,  and note that $C$ is independent of $R$. 
Since $\eta\equiv0$ on~$[0,\tau]$ is admissible by  Hypothesis~\hyperlink{(D)}{{\rm(D)}}, the solution of~\eqref{boussinesq-reduced} with $T(0) = T_0\in X_a$ satisfies the inequality $\|T(\tau)\|_{H^3}\le C$.  Hence,  to prove the validity of~\hyperlink{(H2)}{(H$_2$)},  it suffices to assume $T_0\in X_a$ with
\begin{equation}\label{IC-bounded}
	\|T_0\|_{H^3}\le N, 
\end{equation}
where $N>0$ is a fixed number independent of~$R$ and~$a$.

\smallskip
{\it Step~2: Construction of the target function\/}.  Let $[\e_1,\e_2]\subset (0,c)$ be an interval and let $\chi\in C^\infty(\R)$ be a function such that $\chi(x_3)=T_b$ for $x_3\le\e_1$ and $\chi(x_3)=T_u$ for $x_3\ge\e_2$.  Since 
$$
\chi(x_3)\et=\nabla\biggl(\int_0^{x_3}\chi(s)\,\dd s\biggr) =: \nabla p,
$$
we have $M(\chi)=0$,  where $M$ was defined by~\eqref{velocity-field}. Let us write a solution of~\eqref{boussinesq-reduced} in the form $T=\chi+S$ and note that $M(T) = M(S) + M(\chi) = M(S)$.  Then, $S$~satisfies 
\begin{equation}\label{control-reduced}
	\p_tS-\Delta S+\langle M(S),\nabla\rangle S+M_3(S)\chi'(x_3)=\eta+\chi''(x_3), 
\end{equation}
where $M_3(S)$ stands for the third component of $M(S)$. Furthermore, 
\begin{equation}\label{S-IBV}
	S\bigr|_{\p D}=0, \quad \quad S(0)=S_0:=T_0-\chi. 
\end{equation}
Let $F$ be the space of functions $\zeta\in L_{\rm{loc}}^2(\R_+\times D)$  such that, for any integer $k\ge1$,  the restriction of the vector function $\zeta(k-1+\cdot)$ to~$J$, denoted~$\zeta_k$, belongs to~$E$.  For any integer~$l\ge1$, we define a linear operator $\Pi_l:F\to F$ as follows:
$$
(\Pi_l\zeta)(t,x)=\sum_{k=1}^\infty I_{[k-1,k)}(t)({\mathsf P}_l\zeta_k)(t-k+1,x),
$$
where ${\mathsf P}_l:E\to E$ is the orthogonal projection to the vector span of $\{\varphi_1,\dots,\varphi_l\}$, with $\{\varphi_j\}_{j \geq 1}$ defined in 
\hyperlink{(D)}{{\rm(D)}}. 

Together with~\eqref{control-reduced}, let us consider the problem
\begin{align}
	\p_tv-\Delta v+\langle M(v),\nabla\rangle v+(I-\Pi_l)(M_3(v)\chi')&=(I-\Pi_l)\chi'', \label{v-equation}\\
	v\bigr|_{\p D}=0, \quad v(0)&=v_0.\label{v-IBV}
\end{align}
The proof of the following result is based on a standard contraction mapping argument and is given in Section~\ref{s-stability}. 

\begin{lemma}\label{l-stability}
For any $r_0>0$, there is an integer $l\ge1$ and numbers~$r\ge r_0$, $C>0$, and $\gamma>0$ such that the following properties hold.
	\begin{description}
	\item[Well-posedness.] 
	For any initial condition $v_0\in H_0^1(D)$ satisfying the inequality $\|v_0\|_{H^1}\le r_0$, there is a unique continuous function $v:\R_+\to H_0^1(D)$ satisfying Eqs.~\eqref{v-equation} and~\eqref{v-IBV} in the sense of distributions such that the restriction~$v_k$ of $v(k-1+\cdot)$ to~$J$ belongs to~$\XX^0$ for any $k\ge1$, and 
	\begin{equation} \label{k-bound}
	\|v_k\|_{\XX}\le r, \quad k\ge1. 
	\end{equation}
	\item [Stability.] If $v_0^i\in H_0^1(D)$, $i=1,2$ are two initial conditions satisfying the inequalities $\|v_0^i\|_{H^1}\le r_0$, and $v^i(t,x)$ are the corresponding solutions, then 
\begin{equation}\label{stability-inequality}
	\|v^1(t)-v^2(t)\|_{H^1}\le Ce^{-\gamma t}\|v_0^1-v_0^2\|_{H^1}, \quad t\ge0. 
\end{equation}
	\item [Periodic solution.] Equation~\eqref{v-equation} has a unique continuous $1$-periodic solution $\bar v:\R_+\to H_0^1(D)$ such that, for any $k\ge1$,  the restriction~$\bar v_k$ of $\bar v(k-1+\cdot)$ to~$J$ belongs to~$\XX^0$ and satisfies the inequality $\|\bar v_k\|_\XX\le \frac12r_0$ . 
	\end{description} 
\end{lemma}

We prove that Hypothesis~\hyperlink{(H2)}{(H$_2$)} of Theorem~\ref{t-expomixing} holds with the function $\widehat T:=\chi+\bar v(0)\in\HH$, where $\bar v$ is the periodic solution constructed in Lemma~\ref{l-stability} for $r_0=N+\|\chi\|_{H^1}$ (see \eqref{IC-bounded} for the definition of $N$).  To this end, we show that if $a>0$ is sufficiently large, then there is an integer~$n\ge1$ (depending on~$N$) such that the following property holds:
\begin{itemize}
	\item [\hypertarget{(C)}{\bf(C)}]
\sl 
For any $T_0\in X_a$ satisfying~\eqref{IC-bounded}, there are vectors $\zeta_1,\dots,\zeta_n\in\KK_a$ such that
\begin{align}
	\|\sS^n(T_0;\zeta_1,\dots,\zeta_n)-\widehat T\|_{H^1}
	&\le \frac{1}{12}\,\|T_0-\widehat T\|_{H^1},\label{contraction}\\
	\|\sS^n(T_0;\zeta_1,\dots,\zeta_n)\|_{H^1}&\le N. \label{N-bound}
\end{align} 
\end{itemize}
Once~\eqref{contraction} and~\eqref{N-bound} are proved, the validity of Hypothesis~\hyperlink{(H2)}{(H$_2$)} follows by an iteration.  

Note that if inequality~\eqref{contraction} holds,  then~\eqref{N-bound} is also satisfied, provided that $N\ge 5\|\chi\|_{H^1}$. Indeed, in view of the choice of~$\widehat T$ and~$r_0$, we can write
\begin{align*}
\|\sS^n(T_0;\zeta_1,\dots,\zeta_n)\|_{H^1}&\le \|\sS^n(T_0;\zeta_1,\dots,\zeta_n)-\widehat T\|_{H^1}+\|\widehat T\|_{H^1}\\
&\le\frac{1}{12}\,\|T_0-\widehat T\|_{H^1}+\|\widehat T\|_{H^1}\\
&\le \frac{1}{12}\,\|T_0\|_{H^1}
+\frac{13}{12}\bigl(\|\chi\|_{H^1}+\|\bar v(0)\|_{H^1}\bigr)\\
&\le \frac{1}{12}N+\frac{13}{12}\|\chi\|_{H^1}+\frac{13}{24}r_0
=\frac{15}{24}N+\frac{3}{8}\|\chi\|_{H^1}.
\end{align*}
Assuming without loss of generality that $N\ge 5\|\chi\|_{H^1}$, we see that the  right-most term of this inequality is smaller than~$N$. We thus confine ourselves to the proof of~\eqref{contraction}.

\smallskip
{\it Step~3: Proof of Property~\hyperlink{(C)}{\rm(C)}\/}. Let us fix any $T_0\in X_a$ satisfying~\eqref{IC-bounded} and denote by $v$ the solution of problem~\eqref{v-equation}, \eqref{v-IBV} with $v_0=T_0-\chi$.  The existence and uniqueness of $v$ follows from Lemma~\ref{l-stability} for $r_0=N+\|\chi\|_{H^1} \geq \|T_0 - \chi\|_{H^1}$. 
In view of~\eqref{stability-inequality}, we have
$$
\|v(t)-\bar v(t)\|_{H^1}\le Ce^{-\gamma t}\|v_0-\bar v(0)\|_{H^1}
=Ce^{-\gamma t}\|T_0-\widehat T\|_{H^1}. 
$$
where $\bar v$ stands for the $1$-periodic solution of~\eqref{v-equation} constructed in Lemma~\ref{l-stability}.  If~$t$ is a sufficiently large integer $n\ge1$ (which depends only on~$N$),  then
\begin{equation}\label{T-contraction}
	\|v(n)-\bar v(n)\|_{H^1}
	=\|v(n)+\chi-\widehat T\|_{H^1} 
	\le\frac{1}{12}\|T_0-\widehat T\|_{H^1}. 
\end{equation}
Also,  note that~\eqref{v-equation} can be written as 
$$
\p_tv-\Delta v+\langle M(v),\nabla\rangle v+M_3(v)\chi'=\Pi_l(M_3(v)\chi'-\chi'')+\chi''. 
$$
Since $v(0)=T_0-\chi$,  then~$v$ is the solution of the problem~\eqref{control-reduced}, \eqref{S-IBV} with $\eta=\Pi_l(M_3(v)\chi'-\chi'')$,  and consequently  $S = v$. It follows from~\eqref{k-bound} that the restrictions~$\zeta_k$ of the function $\eta(k-1+\cdot)$ to~$J$ belong to $\lspan\{\varphi_1,\dots,\varphi_l\}$ and have bounded norms depending only on~$N$. Also,  for any ball $B\subset \lspan\{\varphi_1,\dots,\varphi_l\}$ there is $a>0$ such that $B\subset \KK_a$,  and therefore the functions $\zeta_1,\dots,\zeta_n$ belong to~$\KK_a$ for sufficiently large~$a$

We thus conclude that $\sS^n(T_0;\zeta_1,\dots,\zeta_n)=v(n)+\chi$, and the required inequality~\eqref{contraction} follows from~\eqref{T-contraction}. This completes the proof of~\hyperlink{(C)}{(C)}.

\subsection{Density of the image of the derivative}
\label{s-density-derivative}

In this section, we formulate a sufficient condition (see Proposition~\ref{p-UC}),  for Hypothesis~\hyperlink{(H3)}{(H$_3$)} of Theorem~\ref{t-expomixing} to be satisfied. The proof of Proposition~\ref{p-UC} is given in Section~\ref{s-uniquecontinuation}. In what follows, the parameter~$a > 0$ is arbitrary and fixed,  and we write~$X$ instead of~$X_a$.  Our goal is  to prove that for any $T_0\in X$ and $\eta\in E$,  the image of the space $E$ under the mapping $(D_\eta \sS)(T_0,\eta)$ is dense in~$H_0^1(D)$.  Recall that for fixed $\zeta\in E$,  we have $(D_\eta\sS)(T_0,\eta)\zeta= \theta(1)$,   where~$\theta$ is the solution of the linearization of  \eqref{boussinesq-reduced}, 
\begin{equation} \label{linearised-E} 
	\p_t\theta- \Delta\theta+\langle M(T),\nabla \rangle \theta+\langle M(\theta),\nabla\rangle T=\zeta,
\end{equation}
supplemented with the zero initial and boundary conditions:
\begin{align} \label{linearised-IBC}
\theta\bigr|_{\p D}=0, \quad \theta(0)=0.
\end{align}
Here, $T$ is the solution of~\eqref{boussinesq-reduced} with $T(0)= T_0$. For any $\tau\in(0,1)$, we define the vector space 
$$
\Theta_\tau=\{\theta(\tau):\mbox{$\theta$ satisfies~\eqref{linearised-E}, \eqref{linearised-IBC} with some $\zeta\in E$ that vanish on~$[\tau,1]$}\}
$$
in~$H_0^1(D)$, and denote $\Theta:=\Theta_1$. To prove that~$\Theta$ is dense in~$H_0^1(D)$,  we first establish the density in~$L^2(D)$. 

\begin{proposition} \label{p-UC}
For any $T_0\in X$, $\eta\in E$, and $\tau\in(0,1)$, the vector space~$\Theta_\tau$  is dense in~$L^2(D)$. 
\end{proposition}

Taking this result for granted (see Section \ref{s-uniquecontinuation} for a proof),  we establish the density of~$\Theta$ is~$H_0^1(D)$ with help of the following lemma,  which is proved in  Appendix~\ref{a-lemma43}. Note that, in Lemma~\ref{l-density-homogeneous}, we assume zero forcing and vary the initial conditions. 

\begin{lemma} \label{l-density-homogeneous}
For any $T_0\in X$, $\eta\in E$, and $\tau\in(0,1)$, let~$\LL\subset\XX$ be the vector space of solutions for~\eqref{linearised-E} with $\zeta = 0$.  Then the space $\LL_1:=\{\theta(1):\theta\in\LL\}$ is dense in~$H_0^1(D)$.
\end{lemma}

To prove the validity of Hypothesis~\hyperlink{(H3)}{(H$_3$)} of Theorem~\ref{t-expomixing}, fix any $\xi\in H_0^1(D)$, $\e>0$, and $\tau\in(0,1)$. By Lemma~\ref{l-density-homogeneous}, there is a solution $\hat\theta\in\XX$ of Eq.~\eqref{linearised-E} with~$\zeta\equiv0$ such that $\|\hat \theta(1)-\xi\|_{H^1}<\e$. Denoting by~$R(t,\tau)$ the resolving operator for~\eqref{linearised-E}, \eqref{linearised-IBC} with $\zeta \equiv 0$ and initial condition at $t=\tau$,  the semigroup property implies
\begin{equation} \label{closeness}
\|R(1,\tau)\hat\theta_\tau-\xi\|_{H^1}<\e,
\end{equation}
where $\hat\theta_\tau=\hat\theta(\tau)$. Since $R(1,\tau):L^2\to H^1$ is continuous (see~\cite[Section~I.5]{BV1992} for more general results), we can find $\delta>0$ such that inequality~\eqref{closeness} holds with $\hat\theta_\tau=\bar\theta\in L^2$,  provided that $\|\bar\theta-\hat\theta_\tau\|<\delta$. Since~$\Theta_\tau$ is dense in~$L^2$ (see Proposition~\ref{p-UC}), there is a control $\zeta\in E$ vanishing on~$[\tau,1]$ such that the solution~$\theta\in\XX$ of~\eqref{linearised-E}, \eqref{linearised-IBC} satisfies the inequality $\|\theta(\tau)-\hat\theta_\tau\|<\delta$, so that~$\theta(1)$  belongs to the $\e$-neighbourhood of~$\xi$ in~$H^1$. Recalling that~$\xi\in H_0^1$ and~$\e>0$ were arbitrary, we conclude that~$\Theta$ is dense in~$H_0^1$. 

Thus, to complete the proof of Theorem~\ref{MT-mixing}, it remains to establish Proposition~\ref{p-UC}. This is done in Section~\ref{s-uniquecontinuation}.

\section{Density of the image for the linearised problem}
\label{s-uniquecontinuation}
This section is devoted to the proof of Proposition~\ref{p-UC}.  First,  we use a well-known argument to reduce the proof to a unique continuation property for an adjoint problem, involving a temperature~$\psi$, vanishing on the observed domain~$D_c$.  We next show that not only~$\psi$, but also the coresponding velocity~$\www = M(\psi)$ must be zero on~$D_c$. Finally, we apply Fabre--Lebeau estimates to conclude that~$\psi$ (and~$\www$) vanishes identically. 

In what follows, to simplify notation we shall assume $\tau=1$. This does not restrict the generality, since vanishing control on the interval~$[\tau,1]$ does not affect our argument, which is valid on an interval of arbitrary length. 

Throughout the section we assume that $T_0 \in X$, $\eta \in \E$,  and therefore, by Proposition~\ref{p-dissipativity}, we have 
\begin{equation}\label{eq:smT}
\|T(t)\|_{H^3} \leq N\quad \mbox{for all $t \in J$}. 
\end{equation}
Furthermore,  the smoothing of~$M$ implies that 
\begin{equation}\label{eq:smU}
\|\uuu(t)\|_{H^3} =  \|M(T)(t)\|_{H^3} \leq N \quad \mbox{for all $t \in J$},
\end{equation}
where~$N$ depends only on the parameters of the problem and the number~$a$ defining~$X=X_a$ (see~\eqref{set-Xa}).

\subsection{Reduction to unique continuation}
Following a well-known argument in the control theory, together with~\eqref{linearised-E}, we consider the dual equation to  \eqref{linearised-E}:
\begin{equation} \label{dual-linearised}
	\p_t\psi + \Delta\psi+\diver(\psi M(T))-M^*(\psi\nabla T)=0.  
\end{equation} 
Here $M^*$ denotes the (formal) adjoint of~$M$, which takes a function $\fff\in L^2(D,\R^3)$ to the third component of the unique solution $\www=(w_1,w_2,w_3)\in H_0^1\cap H^2$ for the system 
\begin{equation} \label{M-adjoint}
	-\Delta \www+\nabla q=\Ra\,\fff, \quad \diver \www=0,\quad x\in D,
\end{equation}
supplemented with the zero boundary condition.  Since~\eqref{dual-linearised} is equivalent to the system 
\begin{equation} \label{dual-linearised-temperature}
	\p_t\psi+ \Delta\psi+\diver(\psi \uuu)-w_3=0, 
	\quad -\Delta \www+\nabla q=\Ra\,\psi\,\nabla T, 
	\quad \diver \www=0,
\end{equation}
and $T$,  $\uuu$ are smooth (see \eqref{eq:smT} and \eqref{eq:smU}),  energy
arguments similar to those used in Section~\ref{s-EUR} enable one to prove that, for any $\psi_1\in H_0^1(D)$,  problem~\eqref{dual-linearised} has a unique solution $\psi\in\XX$ satisfying the initial condition 
\begin{equation} \label{dual-linearised-IC}
	\psi(1,x)=\psi_1(x). 
\end{equation}
From duality,  for any pair of functions $\theta,\psi\in\XX$ satisfying  equations~\eqref{linearised-E} and~\eqref{dual-linearised}, respectively, we have
\begin{equation} \label{duality}
	\bigl(\theta(1),\psi(1)\bigr)-\bigl(\theta(0),\psi(0)\bigr)=\int_0^1\bigl(\zeta(t),\psi(t)\bigr)\,\dd t. 
\end{equation}
If the subspace~$\Theta$ (defined in Section~\ref{s-density-derivative}) is not dense in~$L^2$, there is a vector $\psi_1 \in L^2$ orthogonal to $\Theta$. Then, by~\eqref{linearised-IBC} and~\eqref{duality}, we have 
\begin{equation} \label{dul}
	\int_0^1\bigl(\zeta(t),\psi(t)\bigr)\,\dd t=0\quad\mbox{for any $\zeta\in E$}. 
\end{equation}
Since $E$ contains all smooth functions supported in $J \times D_c$ (and vanish at $t = 1$),  \eqref{dul}
is equivalent to~$\psi = 0$ on~$J\times D_c$.  Hence, the required density will be established if we prove the following lemma.

\begin{lemma} \label{l-UC}
	Let $T_0\in X$ and~$\eta\in E$ be some functions, let $T\in\XX$ be the solution of~\eqref{boussinesq-reduced} issued from~$T_0$, and let $\psi\in\XX$ be a solution of~\eqref{dual-linearised} vanishing on~$J\times D_c$. Then $\psi\equiv0$. 
\end{lemma}

A proof of this result is given in the next two subsections. We first show that if~$\psi$ vanishes on $J\times D_c$, then the associated velocity field~$\www$ (defined by~\eqref{M-adjoint} with $\fff=\Ra\,\psi\,\nabla T$\,) must be a horizontal flow. We next use some Carleman estimates to prove that~$\psi\equiv0$. 

\subsection{Velocity field on the observed domain}

Let us denote by $(\psi,\www)$ the solution of~\eqref{dual-linearised-temperature} associated with~$\psi\in\XX$ and assume that $\psi = 0$ on~$J\times D_c$.  We also assume that $T$ and $\uuu$
are sufficiently smooth,  but they do not play any role in arguments in this section,  in particular they do not have to solution of any equation.  The goal of this section is to prove that 
 $\www = 0$ on~$J \times D_c$.

First,  we claim that, for $(t,x)\in J\times D_c$, 
\begin{equation} \label{w-linear-inx}
(w_1,w_2)=B(t)x_3+\sum_{k\in \Z_*^2}c_k(t)\sinh\bigl(|k|x_3\bigr)e_k(\tilde x)\,k^\bot,  \qquad w_3 = 0 \,,
\end{equation}
 where $\Z_*^2 = \Z^2 \setminus \{0, 0\}$, 
 $B\in H^1(J,\R^2)$ and $c_k\in H^1(J)$ are some functions, $k^\bot=(-k_2,k_1)$, $\tilde x=(x_1,x_2)$, and $\{e_k\}$ is the trigonometric basis in~$L^2(\T^2)$ defined by 
 $$
 e_k(\tilde x)
 =\left\{
 \begin{array}{cl}
 	\cos(k\cdot\tilde x) & \mbox{for $k_2>0$ or $k_2=0$, $k_1>0$},\\[3pt]
 	\sin(k\cdot\tilde x) & \mbox{for $k_2<0$ or $k_2=0$, $k_1<0$}. 
 \end{array}
 \right.
 $$
Indeed, since~$\psi$ vanishes in~$\DD:=J\times D_c$, the first equation in~\eqref{dual-linearised-temperature} implies that $w_3=0$ in $\DD$.  Combining this with equation for~$w_3$ in~\eqref{dual-linearised-temperature},  we see that~$q_{x_3} = 0$ in $\DD$,  and therefore~$q$ does not depend on~$x_3$ in $\DD$. 
Since $w$ is divergence free, $\tilde{\www}=(w_1,w_2)$ is a divergence-free vector field on~$\T^2_{\tilde x}$ parametrized by~$(t,x_3) \in J \times (0, c)$.  Using the Fourier expansion in~$(x_1, x_2)$, we can write 
\begin{equation} \label{w-prime}
	\tilde{\www}(t,x)=A(t,x_3)+\sum_{k\in\Z_*^2}a_k(t,x_3)e_k(\tilde x)\,k^\bot. 
\end{equation}
Differentiating the second and third equations in~\eqref{dual-linearised-temperature} with respect to~$x_3$ and using the relation $\partial_3 q = \psi = 0$ in~$\DD$,  we conclude, that for any $t \in J$, the functions~$\p_3w_1(t)$ and~$\p_3w_2(t)$ are harmonic in~$x\in D_c$.  Then,  by representation~\eqref{w-prime}, for any $t \in J$ we derive 
$$
A'''=0, \quad a_k'''-|k|^2a_k'=0 \quad\mbox{for $x_3 \in (0, c)$},
$$
where prime denotes the derivative in~$x_3$. Using that~$\www$ vanishes for $x_3=0$, we  obtain
\begin{align*}
A(t,x_3)&=B(t)x_3+C(t)x_3^2,\\	
a_k(t,x_3)&=b_k(t)\bigl(\cosh(|k|x_3)-1\bigr)+c_k(t)\sinh(|k|x_3),
\end{align*}
where $B,C\in H^1(J,\R^2)$ and~$b_k,c_k\in H^1(J)$ are some functions. Thus, to complete the proof of~\eqref{w-linear-inx}, it remains to show that $C\equiv0$ and $b_k\equiv0$. 

To this end, we note that the second and third equations in~\eqref{dual-linearised-temperature} imply, for $t \in J$ and  $\tilde{x} \in \T^2$,
\begin{align}
\p_1 q(t,\tilde x)
&=\Delta w_1= 2C_1(t)-\sum_{k\in\Z_*^2}b_k(t)|k|^2k_2e_k(\tilde x),\label{p_1}\\
\p_2 q(t,\tilde x)
&=\Delta w_2=2C_2(t)+\sum_{k\in\Z_*^2}b_k(t)|k|^2k_1e_k(\tilde x).\label{p_2}
\end{align}
Calculating $\p_1\p_2 q$ from~\eqref{p_1} and~\eqref{p_2} in two different ways, we conclude that $b_k\equiv0$ for any $k\in\Z_*^2$. Substituting these relations into~\eqref{p_1} and~\eqref{p_2}, we obtain $\p_1 q = 2C_1(t)$ and $\p_2 q = 2C_2(t)$. Recalling that~$q$  is periodic in~$x_1$ and~$x_2$, we see that $C_1\equiv C_2\equiv0$, which completes the proof of~\eqref{w-linear-inx}. 

We now show that, after a substitution, we can assume that $\www = 0$ in~$\DD$. Let us introduce vector functions $\zzz=(w_1,w_2,0)$ and $\twww=\www-\zzz$ in $J\times D$, where $(w_1,w_2)$ is defined by the right-hand side of~\eqref{w-linear-inx} in the whole domain $J \times D$. We claim that the triple $(\psi,\twww,q)$ is a solution of~\eqref{dual-linearised-temperature} in~$J\times D$. Indeed, the first equation in~\eqref{dual-linearised-temperature} is satisfied with~$\www$ replaced by~$\twww$, because $z_3 \equiv 0$, and therefore $w_3 = \tilde{w}_3$.  Furthermore,  since  $\zzz$ is harmonic by~\eqref{w-linear-inx},  it follows that 
$$
-\Delta\twww+\nabla q
=-\Delta\www-\Delta\zzz+\nabla q
=\Ra\,\psi\,\nabla T.
$$
Finally, $\twww$ is divergence free, because so are~$\www$ and~$\zzz$. 

Also,  $(\psi,\twww)$ vanishes in~$\DD$,  and therefore $\nabla q = 0$ in $\DD$. Hence, without loss of generality, we assume that $q = 0$. In the next subsection, we prove that $\twww\equiv0$ and $\psi\equiv 0$ in $J\times D$, which  completes the proof of Lemma~\ref{l-UC}.\footnote{Since~$w_1$ and~$w_2$ vanish for $x_3=1$, and we proved that $\twww\equiv0$ in $J \times D$, we can easily conclude that $\www \equiv 0$. However, we do not need this fact.} 

\subsection{Application of Carleman-type estimates}
\label{s-carleman}
In what follows, we drop the tilde from the notation and write $\www$ instead of~$\boldsymbol{\widetilde w}$. We wish to prove that if a solution~$(\psi,\www)$ of~\eqref{dual-linearised-temperature} is zero in~$J\times D_c$,  then~$\psi$ vanishes identically on $J \times D$. The proof of this fact uses some arguments of the paper~\cite{FL-1996} and is divided into four steps. In what follows, we denote by~$\dot I$ the interior of an interval~$I$.

\smallskip
{\it Step~1: Localisation.}
For any $r>0$ and $y^0=(t^0,x^0)\in\R_{t,x}^4$, we denote
\begin{align*}
B_r(y^0)&=\{y = (t,x)\in J \times \R^3:|y-y^0|\le r\}, \\
W_r(y^0)&=\{y=(t,x)\in J \times \R^3:|t-t^0|\le r^2, |x-x^0|\le r\}. 
\end{align*}
We claim that it suffices to prove the following unique continuation property:

\begin{itemize}
	\item [\hypertarget{(P)}{\bf(P)}]
	\sl For any point $y^0=(t^0,x^0)\in\dot J\times D$ with $0<x_3^0<1$ and any smooth function $\sigma(t,x)$ defined in a neighbourhood of~$y^0$ such that $\sigma(y^0)=0$ and $\nabla_x \sigma(y^0)\ne0$, there are~$\e,r\in(0,1)$ with the following property: if $\psi(t,x)=0$ for $(t,x)\in \{y\in\R^4:\sigma(y)<0\}\cap W_r(y^0)$, then $\psi\equiv0$ in~$B_\e(y^0)$.
\end{itemize}
Indeed, assuming that \hyperlink{(P)}{(P)} is established, let us set $\T^2=[0,2\pi]\times[0,2\pi]$ and fix a closed interval $I\subset \dot J$ and a smooth function $g:J\times \T^2\to\R_+$ such that $0<g(t,x')<1$ for $t\in\dot I$, $x'\in \T^2$ and $g(t,x')=0$ for $t\in J\setminus I$, $x'\in \T^2$.  We define
$$
G(\lambda,t,x')=\frac{(1-\lambda)c+\lambda g(t,x')}{1-\lambda+\lambda g(t,x')}, \quad 0\le \lambda<1, \quad (t,x')\in I\times \T^2 \,,
$$
where $c$ originates in the definition of $D_c$.
Thus, $G$ is a smooth function such that $G(0,t,x')\equiv c$ and $G(\lambda,t,x')\nearrow1$ as $\lambda\to1^-$ for $t\in \dot I$. We set 
$$
\hat \lambda=\sup\{\lambda\in[0,1):\psi(t,x)=0\mbox{ for $(t,x_1,x_2)\in I\times \T^2$, $0<x_3<G(\lambda,t,x')$}\}. 
$$
The set under the supremum sign is non-empty as it contains $\lambda = 0$,  since $\psi$ vanishes in $J\times D_c$.  We claim that $\hat\lambda=1$.  For a contradiction suppose $\hat\lambda < 1$ and consider the function $\sigma(t,x)=x_3-G(\hat\lambda,t,x')$. Observe that $\partial_{x_3} \sigma = 1 \neq 0$,  and therefore we can apply Property~\hyperlink{(P)}{(P)} to any point $y^0\in I\times D$ satisfying $\sigma(y^0)=0$. Then, $\psi$ must vanish in a space-time neighbourhood of the surface $\{\sigma=0\}$. This contradicts the definition of~$\hat\lambda$ and shows that~$\psi\equiv0$ in~$I\times D$. Since~$I$ was arbitrary, $\psi \equiv 0$ in~$J \times D$. 

\smallskip
Thus, we need to prove~\hyperlink{(P)}{(P)}. Since~\hyperlink{(P)}{(P)} is a local property, to simplify notation, we assume without loss of generality that $y^0=(0,0)$ and $\Ra=1$. By a rotation of the $x$-coordinate and multiplication of~$\sigma$ by a constant, we can also set $\nabla_x\sigma(0,0)=(0,0,1)$. This transformation does not change the Laplacian, but modifies slightly the divergence and gradient operators, bringing system~\eqref{dual-linearised-temperature} to the form
\begin{equation} \label{dual-transformed}
	\p_t\psi+\Delta\psi+\diver(\psi \vvv)=w_3, 
	\quad -\Delta \www+Aq=\psi AT, 
	\quad A^* \www = 0,
\end{equation}
where $Af = R \nabla f$ and $A^*\zzz = -\diver(R^t\zzz)$, with $R$ being an orthogonal matrix independent of~$(t,x)$, and $\vvv := R^t\uuu$ has the same regularity as~$\uuu$. Observe that~$A$ is a  first-order homogeneous differential operator with constant coefficients and $A^* A = -\Delta$. In the next three steps, we  prove Property~\hyperlink{(P)}{(P)} for \eqref{dual-transformed}, adapting an argument of~\cite{FL-1996}. 

\smallskip
{\it Step~2: Estimating~$\psi$.}
Let~$N$ and~$\delta$ be positive numbers to be chosen below and let~$\varphi\in C_0^\infty(\R^4)$ be defined by~\eqref{function-phi}. We denote by~$t_i$, $r_i$, $C_i$, and~$h_i$ the positive numbers defined in Propositions~\ref{p-carleman-heat} and~\ref{p-carleman-elliptic} and set
$$
t_0=t_1\wedge t_2,\quad 
r_0=r_1\wedge r_2,\quad 
C_0=C_1\vee C_2,\quad 
h_0=h_1\wedge h_2,
$$
so that inequalities~\eqref{carleman-heat} and~\eqref{carleman-laplace} hold for $0<h\le h_0$ and $|t|\le t_0$. Let $r>0$ be so small that $W_r(0)\subset[-t_0,t_0]\times B_{r_0}$ and let $\zeta \in C_0^\infty(W_r(0))$ be a function of the form $\zeta(t,x)=\zeta_1(t)\zeta_2(|x|)$ such that $\zeta_1 (t) = 1$ for $|t| \leq \left(\frac{3}{4}r\right)^2$,  $\zeta_1 (t) = 0$ for $|t| > r^2$, and $\zeta_2 (|x|) = 1$ for $|x| \leq \frac{3}{4}r$,  $\zeta_2 (|x|) = 0$ for $|x| > r$. Notice that $\zeta = 1$ and $\nabla \zeta = 0$ on~$W_{3r/4}(0)$.

Setting $z=\zeta \psi$, we use the first equation in~\eqref{dual-transformed} to write
\begin{equation}\label{cteq}
(\p_t+\Delta)z=\zeta  w_3-z\diver\vvv-\langle\nabla_x z,\vvv\rangle+F,
\end{equation}
where $F=(\p_t\zeta +\Delta\zeta )\psi+2\langle \nabla_x \zeta , \nabla_x\psi\rangle +\psi \langle\nabla_x\zeta ,\vvv\rangle$. Observe that each term in~$F$ contains a derivative of $\zeta$, and consequently $F = 0$ on $W_{3r/4}(0)$. Applying inequality~\eqref{carleman-heat} (cf. Remark \ref{rmk:car}) to $z$ solving \eqref{cteq},   we derive for $0<h\le h_0$
\begin{equation} \label{estimate-psi}
	\bigl\|e^{\varphi/h}z\bigr\|^2+h^2\bigl\|e^{\varphi/h}\nabla_xz\bigr\|^2
	\le C_0h^3\bigl\|e^{\varphi/h}\bigl(\zeta  w_3-z\diver\vvv-\langle\nabla_x z,\vvv\rangle+F\bigr)\bigr\|^2,
\end{equation}
where $\|\cdot\|$ stands for the $L^2$ norm over~$\R^4$. Since~$\uuu$ satisfies \eqref{eq:smU} and~$\vvv$ is a rotation of~$\uuu$, we have $\vvv\in C([-t_0,t_0],H^3)$.  Using the embedding $H^2 \hookrightarrow L^\infty$,  we see that~$\vvv$ and~$\diver\vvv$ are uniformly bounded. It follows that
\begin{equation} \label{esov}
\bigl\|e^{\varphi/h}\bigl(z\diver\vvv+\langle\nabla_x z,\vvv\rangle\bigr)\bigr\|
\le c_1\bigl\|e^{\varphi/h}\bigl(|z|+|\nabla_x z|\bigr)\bigr\|.
\end{equation}
Here and henceforth $c_j$,  $j = 1, 2, \dots$ stand for unessential positive numbers. Substituting~\eqref{esov} into~\eqref{estimate-psi} and decreasing if necessary the number~$h_0$, we obtain
 \begin{equation} \label{estimate-z}
	\bigl\|e^{\varphi/h}z\bigr\|^2+h^2\bigl\|e^{\varphi/h}\nabla_xz\bigr\|^2
	\le c_2h^3\bigl(\|e^{\varphi/h}\zeta  w_3\|^2+\|e^{\varphi/h}F\|^2\bigr),
\end{equation}
where $0<h\le h_0$. Our next goal is to estimate the term involving~$w_3$. 

\smallskip
{\it Step~3: Estimating~$\www$.}
Setting $\yyy=\zeta \www$ and $f=\zeta q$, we use the second equation in~\eqref{dual-transformed} to derive
\begin{equation} \label{equation-for-y}
	-\Delta\yyy +A f=zAT-2\langle\nabla_x\zeta ,\nabla_x\rangle\www-(\Delta\zeta )\www+qA\zeta . 
\end{equation}
Furthermore, applying~$A^*$ to the second equation in~\eqref{dual-transformed} and using the third one, we obtain 
$$
\Delta q= -A^*(\psi A T)=\diver(\psi\nabla_x T).
$$
It follows  that
\begin{equation} \label{Delta-f}
\Delta f-\diver(2q\nabla_x\zeta-z\nabla_xT)
=-\psi\langle\nabla_x\zeta,\nabla_xT\rangle - q\Delta\zeta. 
\end{equation}
We now apply Proposition~\ref{p-carleman-elliptic} twice, to the pair $(y_j,f)$, with the operator $L=A$ (see~\eqref{equation-for-y}), and to the pair $(f,2q\nabla_x \zeta - z\nabla T)$, with the operator $L=\diver$ (see~\eqref{Delta-f}). Setting
\begin{align*}
\GGG_1&= -zAT+2\langle\nabla_x\zeta,\nabla_x\rangle\www+(\Delta\zeta)\www-qA\zeta,\\
G_2&=-\psi\langle\nabla_x\zeta,\nabla_xT\rangle-q\Delta\zeta, \quad 
\boldsymbol{g}=2q\nabla_x\zeta-z\nabla_xT,
\end{align*}
we thus obtain the following two estimates valid for $|t|\le t_0$ and $0<h\le h_0$:
\begin{align*}
	\bigl\| e^{\varphi_t/h}\yyy\bigr\|^2_{x}
	+ h^2\bigl\| e^{\varphi_t/h}\nabla_x\otimes\yyy\bigr\|^2_x
	&\le c_3\Bigl(h\bigl\|e^{\varphi_t/h}f\bigr\|^2_x+h^3\bigl\|e^{\varphi_t/h}\GGG_1\bigr\|_x^2\Bigr), \\
	\bigl\|e^{\varphi_t/h}f\bigr\|_x^2
	+ h^2\bigl\|e^{\varphi_t/h}\nabla_x f\bigr\|_x^2
	&\le c_3\Bigl(h\bigl\|e^{\varphi_t/h}\boldsymbol{g}\bigr\|_x^2+h^3\bigl\|e^{\varphi_t/h}G_2\bigr\|_x^2\Bigr),
\end{align*}
where $\|\cdot\|_x$ stands for the $L^2$ norm over~$\R^3$ (note that all the terms in the right-hand sides are supported on~$W_r(0)$). 
By \eqref{eq:smT},  the function~$T$ is uniformly (in time) bounded in $H^3$,  and consequently, by the embedding $H^2 \hookrightarrow L^\infty$, we see that~$\nabla T$ and~$AT$ are uniformly bounded in space and time. Substituting the second inequality into the first one, integrating in time, rearranging terms between $\GGG_1$ and $G_2$, we obtain the following estimate for the third component of~$\yyy$:
\begin{equation} \label{estimate-y3}
	\bigl\|e^{\varphi/h}\zeta w_3\bigr\|^2
	\le c_4h^2\Bigl(\bigl\|e^{\varphi/h}\GGG\bigl\|^2
	+\bigl\|e^{\varphi/h}z\bigl\|^2\Bigr),
\end{equation}
where the function~$\GGG$ collects all the terms containing derivatives of~$\zeta $ in~$x$:
$$
\GGG = |q\nabla_x\zeta| + |\langle\nabla_x\zeta, \nabla_x\rangle\www| + |(\Delta\zeta )\www|  + |\psi \langle \nabla_x\zeta,  \nabla_xT\rangle| + |q\Delta\zeta|.
$$
Combining~\eqref{estimate-z} and~\eqref{estimate-y3}, for sufficiently small~$h$ we derive the following inequality after rearrangement:
\begin{equation} \label{final-estimate}
	\bigl\|e^{\varphi/h}z\bigr\|^2+h^2\bigl\|e^{\varphi/h}\nabla_xz\bigr\|^2
	\le c_5h^3\,\bigl\|e^{\varphi/h}\bigl(|F|+\GGG\bigr)\bigr\|^2. 
\end{equation}

\smallskip
{\it Step~4: Conclusion of the proof.}
Let us set $\Sigma =\supp(\nabla\zeta )\cap\{\sigma\ge0\}$ and suppose that, for $\varphi$ defined in \eqref{function-phi}, we have proved the inequality
\begin{equation} \label{phi00}
\varphi(0,0)=\delta^2>\sup_{(t,x)\in \Sigma }\varphi(t,x)=:M
\end{equation}
for an appropriate choice of the parameters~$N$ and~$r$. Then, we can find $\e>0$ such that 
$$
m:=\inf_{(t,x)\in B_\e(0)}\varphi(t,x)\ge \e+\sup_{(t,x)\in \Sigma }\varphi(t,x) = \e + M \,. 
$$
Since each term in $F$ or $|\GGG|$ contains 
a derivative of $\zeta$, we have $|F|+|\GGG|=0$ outside~$\Sigma $, and therefore~\eqref{final-estimate} with $|h|\ll1$ implies that
$$
e^{2m/h}\|z\|_{L^2(B_\e(0))}^2
\le \bigl\|e^{\varphi/h}z\bigr\|^2
\le c_5h^3e^{2M/h}\bigl(\|F\|^2+\|\GGG\|^2\bigr). 
$$
Consequently,
$$
\|z\|_{L^2(B_\e(0))}^2\le c_5h^3e^{-2\e/h}\bigl(\|F\|^2+\|\GGG\|^2\bigr). 
$$
Letting $h \to 0^+$ and using that~$F$ and~$\GGG$ are uniformly bounded in $h$, we see that~$z$ and, hence, $\psi$~must vanish in~$B_\e(0)$, as desired. Thus, to complete the proof, it remains to establish~\eqref{phi00} for an appropriate choice of~$N$, $\delta$, and~$r$. 

Let us set $\varphi_1(t,x)=\delta-x_3-N(x_1^2+x_2^2+t^2)$ and observe that our goal is to prove the inequality
\begin{equation} \label{reduced-inequality}
	\sup_{(t,x)\in\Sigma }\varphi_1(t,x)<\delta. 
\end{equation}
By the definition of~$\zeta$, if $(t,x)\in \Sigma \subset \supp(\nabla\zeta)$, then 
\begin{equation} \label{support-inx}
\frac{3r}{4}\le |x|\le r, \quad |t|\le r^2
\qquad\mbox{or}\qquad 
|x|\le r, \quad \Bigl(\frac{3r}{4}\Bigr)^2\le |t|\le r^2. 
\end{equation}
To prove~\eqref{reduced-inequality}, we first assume that $x_3\ge0$. If $x_1^2+x_2^2\ge r^2/2$, then $\varphi_1(t,x)\le \delta - Nr^2/2$. Otherwise, it follows from~\eqref{support-inx} that either $x_3\ge r/4$ or $|t|\ge (\frac{3r}{4})^2$. In both case, we obtain $\varphi_1(t,x)\le \delta-\lambda(r)$, where $\lambda(r)$ is the minimum of~$r/4$ and~$(3r/4)^2$, so that~\eqref{reduced-inequality} follows. 

Next, we assume $x_3<0$.  Since $\sigma(0)=0$ and $\nabla_x\sigma(0)=(0,0,1)$,  there is a number $C\ge1$,  independent of~$\delta$, $r$, and~$N$ (in fact $C$ depends only on $\sigma$), 
such that 
$$
\sigma(t,x)\le \frac{x_3}{2}+C(x_1^2+x_2^2+|t|)\quad\mbox{for $|t|+|x|\le C^{-1}$}. 
$$
Since $\sigma \geq 0$ on $\Sigma$, it   follows that 
\begin{equation} \label{estimate-forx3}
-x_3\le 2C(x_1^2+x_2^2+|t|)\quad\mbox{for $(t,x)\in\Sigma$},	
\end{equation}
provided  $r\ll1$. We thus obtain the inequality 
\begin{equation} \label{bound-phi1-upper}
\varphi_1(t,x)\le  
\delta-(Nt^2-2C|t|)-(N-2C)(x_1^2+x_2^2). 
\end{equation}
Up to this point, the number~$N$ was arbitrary and $r>0$ was sufficiently small. We now fix a positive number $r<\frac{1}{12C}$ for which the above arguments are valid and set $N=10C/r^2$.

Suppose $(t,x)\in\Sigma$ is such that the first pair of inequalities in~\eqref{support-inx} holds. In this case, if $x_1^2+x_2^2\ge r^2/2$, then~\eqref{bound-phi1-upper} implies that  
$$
\varphi_1(t,x)\le \delta+2Cr^2-8C(x_1^2+x_2^2)\le \delta-2Cr^2. 	
$$
If $x_1^2+x_2^2\le r^2/2$,  then~\eqref{estimate-forx3} shows that  $|x_3|\le 3Cr^2$. On the other hand, the first inequality in~\eqref{support-inx} implies that 
$$
|x_3|^2\ge \frac{9r^2}{16}-x_1^2-x_2^2\ge \frac{r^2}{16},
$$
so that $|x_3|\ge r/4$, and we obtain a contradiction with the bound $|x_3| \leq 3Cr^2$ for $r\le \frac{1}{12C}$. 

Finally, assume that $(t,x)\in\Sigma$ is such that the second pair of inequalities in~\eqref{support-inx} holds. In this case, it follows from~\eqref{bound-phi1-upper} and the last inequality in~\eqref{support-inx} that 
$$
\varphi_1(t,x)\le \delta-|t|(N|t|-2C)\le \delta-2Cr^2. 
$$
This completes the proof of~\eqref{reduced-inequality} and that of Property~\hyperlink{(P)}{(P)}.

\section{Appendix}
\subsection{Sufficient condition for exponential mixing}

In this section we state an abstract result that guarantee mixing and uniqueness of the invariant measure.  We also use notation from Section~\ref{s0}.

Let $X$ be a compact subset of a closed affine subspace~$\HH$ in a separable Hilbert space~$H$.  Let $(T_k,\IP_T)$ be a discrete-time Markov process in~$X$ with a transition function $P_k(T,\Gamma)$ 
(where $T\in X$ and $\Gamma\in\BB(X)$) and the corresponding Markov operators~$\PPPP_k$ and~$\PPPP_k^*$ acting in the spaces~$C(X)$ and~$\PP(X)$, respectively. We assume that the Markov process 
 is generated by a random dynamical system of the form~\eqref{RDS}, where $\{\eta_k\}_{k\ge1}$ is a sequence of i.i.d.\ random variables in a separable Hilbert space~$E$, and $\sS:\HH\times E\to \HH$ is a continuous map. We denote by~$\KK$ the support of the law for~$\eta_k$ and assume that $\sS(X\times \KK)\subset X$. A proof of the following result can be found in~\cite{JNPS-2021,shirikyan-jems2021}.

\begin{theorem} \label{t-expomixing}
	Let us assume that the random dynamical system~\eqref{RDS} satisfies the following four hypotheses. 
\begin{itemize}
	\item[\hypertarget{(H1)}{\bf(H$_1$)}] 
	There is a Hilbert space~$V$ compactly embedded into~$H$ such that the map~$\sS$ is twice continuously differentiable from~$\HH\times E$ to~$V$, and its derivatives are bounded on bounded subsets.
	\item[\hypertarget{(H2)}{\bf(H$_2$)}]
	For any $\e>0$, there is a function  $\widehat T\in X$ and an integer $m\ge1$ with the following property: for any $T_0\in X$ one can find $\zeta_1,\dots,\zeta_m\in\KK$ such that 
	\begin{equation} \label{e-control}
		\|\sS^m(T_0;\zeta_1,\dots,\zeta_m)-\widehat T\|_H\le \e,
	\end{equation}
	where $\sS^k(T_0;\eta_1,\dots,\eta_k)$ denotes the trajectory of~\eqref{RDS} at time~$k$. 
	\item[\hypertarget{(H3)}{\bf(H$_3$)}] 
	For any $T\in X$ and~$\eta\in\KK$, the derivative $(D_\eta \sS)(T,\eta):E\to H$ has a dense image. 
	\item[\hypertarget{(H4)}{\bf(H$_4$)}] 
	There is an orthonormal basis $\{e_j\}$ such that the random variables~$\eta_k$ can be written in the form
	\begin{equation*}
		\eta_k=\sum_{j=1}^\infty b_j\xi_{jk}e_j,
	\end{equation*}
	where $b_j$ are non-zero numbers such that $\sum_jb_j^2<\infty$, and~$\xi_{jk}$ are independent random variables whose law~$\ell_j$ possess densities~$\rho_j\in C^1(\R)$ supported in the interval~$[-1,1]$.
\end{itemize}
Then the Markov process~$(T_k,\IP_T)$ has a unique stationary measure $\mu\in\PP(X)$, and there are positive numbers~$C$ and~$\gamma$ such that, for any $\lambda\in\PP(X)$, 
\begin{equation} \label{mixing-dL}
	\|\PPPP_k^*\lambda-\mu\|_L^*\le Ce^{-\gamma k}\quad\mbox{for all $k\ge0$},
\end{equation}
where $\|\cdot\|_L^*$ stands for the dual-Lipschitz norm on~$\PP(X)$. 
\end{theorem}

\subsection{Proof of Lemma~\ref{l-stability}}
\label{s-stability}
{\it Step~1: Preliminary estimates\/}. Given a function $f:\R_+\times D\to \R$ and number $s\ge0$, we shall write $\theta_sf$ for the restriction of its translation $t\mapsto f(s+t,x)$ to the domain $J\times D$. 
Let us first consider the equation
\begin{equation}\label{v-boussinesq}
	\p_t v-\Delta v+\langle M(v),\nabla\rangle v=g(t,x),
\end{equation}
supplemented with the initial and boundary conditions~\eqref{v-IBV}.  Using energy estimates,  it is standard to prove that \eqref{v-boussinesq} is well posed in~$H_0^1(D)$ for any 
$g\in L_{\rm{loc}}^2(\R_+\times D)$.  We claim that if $\delta>0$ is sufficiently small and 
\begin{equation}\label{g-small}
	\sup_{t\ge0}\|\theta_t g\|_{L^2(J\times D)}\le\delta,
\end{equation}
then, for any $R>0$ and a sufficiently large $c(R)>0$, the following properties hold:
\begin{itemize}
	\item [\hypertarget{(a-bound)}{\bf(a)}]
\sl If $v_0\in H_0^1(D)$ is such that $\|v_0\|_{H^1}\le R$, then the solution~$v$ of~\eqref{v-boussinesq}, \eqref{v-IBV} satisfies the inequality
\begin{equation}\label{v-estimate}
	\sup_{t\ge 0}\|\theta_tv\|_\XX\le c(R). 
\end{equation}
	\item [\hypertarget{(b-bound)}{\bf(b)}]
If $v_0^1, v_0^2\in H_0^1(D)$ are two initial conditions such that $\|v_0^i\|_{H^1}\le R$, $i=1,2$, and~$g_i\in L_{\mathrm{loc}}^2(\R_+\times D)$ are two functions for which~\eqref{g-small} holds, then the corresponding solutions satisfy the inequality
\begin{equation}\label{v1-v2}
	\|\theta_t v^1-\theta_t v^2\|_\XX
	\le c(R)e^{-\gamma t}\Bigl(\|v_0^1-v_0^2\|_{H^1}+\sup_{0\le s\le t}\bigl(e^{\gamma s}\|\theta_s g_1-\theta_s g_2\|_{L^2}\bigr)\Bigr),
\end{equation}
where $t\ge0$ is arbitrary, $\gamma>0$ is a number not depending on~$R$, the initial conditions $v_0^i$,  $i = 1,  2$,  or the right-hand sides $g_i$,  $i = 1,  2$,  
and the $L^2$ norm is taken on the domain~$J\times D$. 
\end{itemize}
These properties are well known for the 2D Navier--Stokes equations (cf.~\cite[Chapters~I and~II]{BV1992}).  For~\eqref{v-boussinesq}, the proof is similar and can be completed with the help of the arguments used for Propositions~\ref{p-EU} and~\ref{p-dissipativity}. 

\smallskip
{\it Step~2: Uniqueness\/}. Let $v^1, v^2$ be two solutions of~\eqref{v-equation}, \eqref{v-IBV} satisfying~\eqref{k-bound}. Since $\XX$ is continuously embedded into~$L^\infty(J,  H^1(D))$,  we see that 
\begin{equation}\label{v10-v20}
	\|v^i(0)\|_{H^1}\le R, \quad i=1,2,
\end{equation}
where $R$ depends only on~$r$. The functions~$v_i$ can be regarded as solutions of~\eqref{v-boussinesq} with the right-hand sides 
\begin{equation} \label{g_i}
g_i(t)=(I-\Pi_l)(\chi''-M_3(v^i)\chi').	
\end{equation}
Since $v_i$, $i = 1,  2$ satisfies~\eqref{k-bound}, we can find an integer $l_0\ge1$ depending only on~$r$ such that~\eqref{g-small} holds for~$g_i$ from \eqref{g_i},   provided that $l\ge l_0$. Setting $v=v^1-v^2$ and $g=g_1-g_2$, we use assertion~\hyperlink{(b-bound)}{(b)} to write
\begin{equation} \label{bound-v-g}
	\sup_{t\ge0}\|\theta_tv\|_\XX\le c(R)\sup_{t\ge0}\|\theta_tg\|_{L^2(J\times D)},
\end{equation}
where the suprema on both sides are taken over all non-negative integers $t$.  Since $g=-(I-\Pi_l)M_3(v)\chi'$,   it follows that 
\begin{equation}\label{bound-g-l}
	\sup_{t\ge0}\|\theta_tg\|_{L^2(J\times D)}\le \delta_l\sup_{t\ge0}\|\theta_tv\|_\XX,
\end{equation}
where $\{\delta_l\}$ is a sequence converging to zero as $l\to\infty$. Combining~\eqref{bound-v-g} and~\eqref{bound-g-l}, and choosing~$l\ge1$ so large that $\delta_lc(R)\le\frac12$, we see that $v\equiv0$. 

\smallskip
{\it Step~3: Existence\/}. Let us fix an integer $n$ and an initial condition $v_0\in H_0^1(D)$ such that $\|v_0\|_{H^1}\le r_0$. We use a fixed point argument to construct a solution~$v$ of~\eqref{v-equation}, \eqref{v-IBV} on every time interval $J_n=[0,n]$ with integer $n \geq 1$ such that~\eqref{k-bound} holds for $\theta_kv=v_k$ with $0\le k\le n-1$ and a number~$r>0$ not depending on~$n$.  Then, the existence of the global solution follows from the uniqueness proved in Step 2.  

For each $r > 0$, denote
\begin{equation*}
\FF_n(r) = \{w\in L^2(J_n \times D) : \|\theta_kw\|_{L^2(J \times D)}\le r  \textrm{ for all integers } 0\le k\le n-1 \} 
\end{equation*}
and note that~$\FF_n(r)$  is a closed subset in the Hilbert space~$L^2(J_n \times D)$. For any $w \in \FF_n(r)$, we define\footnote{Note that the operator~$\Pi_l$ initially defined on functions of $(t,x)\in \R_+\times D$ can be extended in a natural manner to functions on~$J_n\times D$.}  $g=g_w:=(I-\Pi_l)(\chi''-M_3(w)\chi')$. By 
the smoothing of $M$,  smoothness of $\chi$,  
the inverse Poincar\'e inequality, for every $\delta > 0$ implies that there is $l=l(r) \geq 0$ such that \eqref{g-small} holds. By Step~1, there exists a map  $R_{v_0}:\FF_n(r)\to\XX_n^0$ taking $w \in \FF_n(r)$ to the solution $v\in\XX_n^0$ of problem~\eqref{v-boussinesq}, \eqref{v-IBV}. Since the embedding $\XX_n^0 \hookrightarrow L^2(J_n \times D)$ is compact,  the map $R_{v_0}:\FF_n(r)\to L^2(J_n \times D)$ is compact.  In addition,  by~\eqref{v-estimate}, if~$r \geq c(r_0)$ is sufficiently large, then~$R_{v_0}$ maps the set~$\FF_n(r)$ into itself.  The Leray--Schauder theorem now implies that~$R_{v_0}$ has a fixed point, which is the required solution of~\eqref{v-equation}, \eqref{v-IBV} on $J_n$.

\smallskip
{\it Step~4: Exponential stability\/}. Let $v^1$ and~$v^2$ be two solutions of~\eqref{v-equation}, \eqref{v-IBV}, \eqref{k-bound} associated with some initial functions $v_0^1,v_0^2\in B_{H_0^1}(r_0)$. Regarding~$v^i$ as solutions of~\eqref{v-boussinesq} with~$g_i$ given by~\eqref{g_i}, we can write inequality~\eqref{v1-v2} for the difference $v=v^1-v^2$ as
\begin{equation}\label{sup-v-t}
\sup_{0\le s\le t}\bigl(e^{\gamma s}\|\theta_s v\|_\XX\bigr)
\le c(R)\|v_0\|_{H^1}+c(R) \sup_{0\le s\le t}\bigl(e^{\gamma s}\|\theta_s g\|_{L^2(J\times D)}\bigr),	
\end{equation}
where $g=g_1-g_2$, $v_0=v_0^1-v_0^2$, $t\ge0$ is any integer, and the supremums on both sides are taken over all integers~$s\in[0,t]$. Exactly the same argument as in Step~2 allows one to absorb the second term on the right-hand side of~\eqref{sup-v-t} by the left-hand side if $l$ is sufficiently large.   We thus obtain
$$
\sup_{0\le s\le t}\bigl(e^{\gamma s}\|\theta_s v\|_\XX\bigr)\le C\|v_0\|_{H^1},
$$
provided that~$l$ is sufficiently large. This implies the required inequality~\eqref{stability-inequality}. 

\smallskip
{\it Step~5: Periodic solution\/}. Let us fix any $r_0>0$. To construct a periodic solution with an $\XX$-norm smaller than~$r_0/2$, we first show that if $r>0$ is sufficiently large, then there is an integer $l_r\ge1$ such that, for any $l\ge l_r$,  the equation~\eqref{v-equation} has a unique $1$-periodic solution~$\bar v^l$ such that~\eqref{k-bound} holds for it. The uniqueness follows immediately from~\eqref{stability-inequality}, so that we only prove the existence. 

If $l$ is sufficiently large,  there is $\tilde{r} > 0$ and ~$\tilde v$ the solution of~\eqref{v-equation}, \eqref{v-IBV} with $v_0=0$ such that $\|\theta_k \tilde v\|_\XX\le \tilde r$ for any $k\ge1$.  Then by~\eqref{stability-inequality},  for any $r>0$, there is an integer $l_r\ge1$ such that, for any $v_0\in B_{H_0^1}(r)$ and any integer $l\ge l_r$, the unique solution of~\eqref{v-equation}, \eqref{v-IBV} constructed in Steps~2 and~3 satisfies the inequality 
\begin{equation}\label{large-time}
\|v(t)-\tilde v(t)\|_{H^1}\le Ce^{-\gamma t}\|v_0\|_{H^1}\le Ce^{-\gamma t}r.
\end{equation}
In particular, if $r\geq 2\tilde{r}$ and $t=n\ge1$ is a large integer, then $\|v(n)\|_{H^1}\le r$. Denoting by~$\RR_t:B_{H_0^1}(r)\to H_0^1$ the map taking~$v_0$ to the value of the corresponding solution of~\eqref{v-equation}, \eqref{v-IBV} at time~$t$, we conclude that~$\RR_n$ maps the ball~$B_{H_0^1}(r)$ into itself. Using again~\eqref{stability-inequality}, we see that~$\RR_n$ is a contraction for any $n\ge n_r$, where $n_r$ is a sufficiently large integer. By the Banach fixed point theorem, there exists the unique fixed point $w^n\in B_{H_0^1}(r)$.  If $w^n$ and $w^{n+1}$ are the fixed points corresponding for~$\mathcal{R}_n$ and~$\mathcal{R}_{n+1}$,  then both of them are fixed points corresponding to~$\mathcal{R}_{n(n+1)}$.  However, since the latter is unique,  we obtain $w^n = w^{n+1}$. Therefore $\mathcal{R}_1(w^n) = w^{n+1}(n+1)=w^{n+1}(0)=w^n$, so that~$\RR_t(w^n)$ is a 1-periodic solution of \eqref{v-equation}, which will be denoted by~$\bar v^l$. 

To complete the construction of the required periodic solution of~\eqref{v-equation}, it suffices to show that 
\begin{equation}\label{v-r-0}
	\|\bar v^l\|_{\XX}\to0\quad\mbox{as $l\to\infty$}. 
\end{equation}
Once this is proved, we can choose $l$ so large that~$\|\bar v^l\|_\XX\le r_0/2$. 

To prove~\eqref{v-r-0}, we view~$\bar v^l$ as a solution of~\eqref{v-boussinesq} and use inequality~\eqref{v1-v2} with $v^2\equiv0$ and the periodicity of~$\bar v^l$ to write
$$
\|\bar v^l\|_\XX=\|\theta_n\bar v^l\|_\XX
\le Ce^{-\gamma n}\|\bar v^l(0)\|_{H^1}
+C\|g^l\|_{L^2(J\times D)},
$$
where $g^l$ is given by relation~\eqref{g_i} in which~$v^i$ is replaced by~$\bar v^l$. Since $n\ge1$ is arbitrary and $\|g^l\|_{L^2(J\times D)}\to0$ as $l\to\infty$, the above inequality readily implies~\eqref{v-r-0}. This completes the proof of the lemma. 

\subsection{Proof of Lemma~\ref{l-density-homogeneous}}
\label{a-lemma43}
We follow a standard argument based on backward uniqueness for~\eqref{linearised-E}; cf.~\cite[Section~7.2]{KNS-gafa2020}. 
For the rest of the proof, we fix any $T_0 \in X$ and $\eta \in E$ and denote by $T \in \XX$ the solution of~\eqref{boussinesq-reduced} constructed in Proposition \ref{p-EU}.  Recall that $R(t,\tau)$ denotes the resolving operator for equation~\eqref{linearised-E} with $\zeta \equiv 0$ and an initial condition specified at time $t=\tau$.  Specifically, for given $\tau, t \in \R$ satisfying the inequality $\tau<t$ and for any $\theta_\tau \in H^1_0(D)$,  we have $R(t,\tau)\theta_\tau = \theta(t)$,  where $\theta \in \XX$ is a solution of~\eqref{linearised-E} with $\theta(\tau) = \theta_\tau$.  It is well known (cf.\ Proposition~\ref{p-EU}) that $R(t,\tau): H^{1}_0 \to H^{1}_0$ is well-defined continuous linear map.  

Suppose that~$\LL_1$ is not dense in~$H_0^1$. Then, there is a non-zero $\psi_1\in H^{-1}$ such that 
\begin{equation} \label{orthogon}
\bigl(R(1,0)\theta_0,\psi_1\bigr)=0
	\quad\mbox{for any $\theta_0\in H_0^1$},
\end{equation}
where $(\cdot,  \cdot)$ is understood as the duality pairing between~$H_0^1$ and~$H^{-1}$. Following a well-known idea, let us consider the dual problem~\eqref{dual-linearised}, \eqref{dual-linearised-IC}. We claim that it has a unique solution~$\psi$ of the form
\begin{equation} \label{psi}
\psi(t)=\bar\psi(t)+\xi(t),	\quad \bar\psi(t):=e^{(1-t)\Delta}\psi_1,
\end{equation}
where $\xi\in L^2(J,H^1)\cap H^1(J,H^{-1})$. Indeed,  the uniqueness is standard and follows from energy estimates.  To prove the existence, we substitute~\eqref{psi} into~\eqref{dual-linearised} and~\eqref{dual-linearised-IC} and derive the following problem for~$\xi$:
\begin{equation} \label{problem-xi}
	\p_t \xi+ \Delta \xi+\diver(\xi M(T))-M^*(\xi\nabla T)=g, \quad \xi(1)=0 \,, 
\end{equation}
where $g:=M^*(\bar\psi\nabla T)-\diver(\bar\psi M(T))$.  By the regularising property of~$M$ and Sobolev embeddings,  we have $g\in L^2(J,H_0^1)+L^2(J,H^{-1})$.  Then,  by a fixed point argument (cf.\ proof of Lemma \ref{l-stability}), it follows that we can find a unique solution~$\xi$ of~\eqref{problem-xi} in the space $L^2(J,H_0^1) \cap H^1(J,H^{-1})$. 
 
Now note that $\frac{\dd}{\dd t}(\theta(t),\psi(t))=0$ for $t\in(0,1)$, so that the function $(\theta(t),\psi(t))$ does not depend on~$t\in[0,1]$. Combining this with~\eqref{orthogon}, we see that  
$$
0=\bigl(R(1,0)\theta_0,\psi_1\bigr)
=\bigl(\theta(1),\psi(1)\bigr)
=\bigl(\theta_0,\psi(0)\bigr)
$$
for any $\theta_0\in H_0^1$. It follows that $\psi(0)=0$. By the backward uniqueness for~\eqref{dual-linearised} (see e.g.~\cite[Section~II.8]{BV1992}),  we conclude that $\psi_1=0$. This contradicts the hypothesis that $\psi_1\ne0$ and completes the proof of the lemma. 

\subsection{Carleman estimates of Fabre--Lebeau}
\label{s-Carlemanestimates}
In what follows, we denote by~$|\cdot|$ the maximum norm on Euclidean spaces, and given $r>0$, we write $B_r=\{x\in\R^3:|x|<r\}$. For positive numbers $t_0$, $r_0$, $N$, and~$\delta$ define 
\begin{align} 
W(t_0,r_0)&=\{(t,x)\in\R^4:|t|\le t_0,|x|\le r_0\}, \label{set-W}\\
\varphi(t,x)&=\bigl(x_3+N(x_1^2+x_2^2+t^2)-\delta\bigr)^2\chi(t,x),	
\label{function-phi}
\end{align}
where~$\chi$ is a smooth function with compact support that is equal to~$1$ on~$W(1,1)$.  If we consider~$\varphi$ as a function of~$x$ and regarding~$t$ as a parameter,   we write 
$\varphi_t(x)$ instead of $\varphi(t, x)$.  We recall two Carleman-type estimates for elliptic and parabolic problems established by Fabre and Lebeau~\cite{FL-1996} (see also~\cite{LR-1995} for some related results).

\subsubsection*{Heat operator}
The following result is~\cite[Lemma~4.3]{FL-1996} applied to the weight function $\varphi$ given by~\eqref{function-phi}. The fact that $\varphi$ satisfies the
required hypotheses  is established in \cite[Lemma~4.5]{FL-1996} when $N=1$. The general case  follows by a similar argument.   

\begin{proposition} \label{p-carleman-heat}
For any $N>0$ and sufficiently small $\delta>0$, there are positive numbers $t_1$, $r_1$, $C_1$, and~$h_1$ such that, for any function 
	\begin{equation*} 
		z\in L^2\bigl([-t_1,t_1],H_0^2(B_{r_1})\bigr)\cap H_0^1\bigl([-t_1,t_1],L^2(B_{r_1})\bigr),
	\end{equation*}
and for any $0<h\le h_1$ one has 
	\begin{equation} \label{carleman-heat}
	\bigl\|e^{\varphi/h}z\bigr\|^2+h^2\bigl\|e^{\varphi/h}\nabla_xz\bigr\|^2\le C_1h^3\bigl\|e^{\varphi/h}(\p_t-\Delta)z\bigr\|^2, 
	\end{equation}
	where $\|\cdot\|$ stands for the $L^2$-norm taken over~$\R^4$. 
\end{proposition}

\begin{remark}\label{rmk:car}
Inequality~\eqref{carleman-heat} is invariant under the time reversal, so that the heat operator entering its right-hand side can be replaced by the backward heat operator $\p_t+\Delta$.
\end{remark}

\subsubsection*{Elliptic operators with constant coefficients}
We now focus on Stokes-type systems satisfied by the velocity field in the Boussinesq system. Since we investigate stationary solutions, we regard~$t$ as a parameter. Let~$L$  be a first-order differential operator acting on vector functions:
\begin{equation*} 
	Lf=\sum_{j=1}^3\sum_{k=1}^n a_{jk}\p_j f_k, \quad f=(f_1,\dots,f_n),
\end{equation*}
where $a_{jk}$ are given real numbers. The following result is a consequence of~\cite[Theorem~3.1]{FL-1996} applied to the function~$\varphi_t$ defined in \eqref{function-phi} with a sufficiently small~$t$. 

\begin{proposition} \label{p-carleman-elliptic}
For any $N>0$ and sufficiently small $\delta>0$, there are positive numbers $t_2$, $r_2$, $C_2$, and~$h_2$ such that for any functions $y\in H_0^1(B_{r_2})$ and $f\in L^2(\R^3)^n$ satisfying 
	\begin{equation*} 
	 \supp f\subset B_{r_2},\quad 
	 \Delta y-Lf\in L^2(B_{r_2}),\quad 
	\supp(\Delta y-Lf)\subset B_{r_2}, 
	\end{equation*}
the following inequality holds for $0<h\le h_2$ and $|t|\le t_2:$
	\begin{equation} \label{carleman-laplace}
	\bigl\|e^{\varphi_t/h}y\bigr\|^2+h^2\bigl\|e^{\varphi_t/h}\nabla_xy\bigr\|^2
	\le C_2\Bigl(h\bigl\|e^{\varphi_t/h}f\bigr\|^2+h^3\bigl\|e^{\varphi_t/h}(\Delta y-Lf)\bigr\|^2\Bigr), 
	\end{equation}
	where $\|\cdot\|$ stands for the $L^2$-norm taken over~$\R^3$. 
\end{proposition}

\begin{proof}
Let us set
	$$
	a_t(x,\xi)=|\xi|^2-|\nabla\varphi_t(x)|^2, \quad 
	b_t(x,\xi)=2\langle\xi,\nabla\varphi_t(x)\rangle.
	$$
Recall that the Poisson bracket of~$a_t$ and~$b_t$ is defined by
$$
\{a_t,b_t\}(x,\xi)=\sum_{j=1}^3
\biggl(\frac{\p a_t}{\p\xi_j}\frac{\p b_t}{\p x_j}-\frac{\p a_t}{\p x_j}\frac{\p b_t}{\p\xi_j}\biggr). 
$$
By~\cite[Theorem~3.1]{FL-1996}, it suffices to prove that, for any $(t,x,\xi)\in\R^7$ such that $|t|+|x|\ll1$, the relations $a_t(x,\xi)=0$ and $b_t(x,\xi)=0$ imply the inequality
\begin{equation} \label{Poisson-lowerbound}
	\{a_t,b_t\}(x,\xi)\ge c,
\end{equation}
where $c>0$ does not depend on $(t,x,\xi)$. To this end, assume that $\delta > 0$  in the definition of $\varphi_t$ is small and $|x|+|t|\le \delta/2$.
For such $(x, t)$ we have $\varphi_t \approx \delta^2$ which means $C^{-1} \delta^2 \leq \varphi_t(x, t) \leq C\delta^2$ for some universal constant $C>1$. 
 It is straightforward to check that 
\begin{equation} \label{PB}
	\{a_t,b_t\}(x,\xi)=4\sum_{j,k=1}^3\p_j\p_k\varphi_t\bigl(\xi_j\xi_k+\p_j\varphi_t \p_k\varphi_t \bigr).
\end{equation}
Next, note that
$$
\p_j\varphi_t(x)=2(x_3-\delta)\delta_{j3} +  O(\delta^2), \quad
\p_j\p_k\varphi_t(x)=2\delta_{j3}\delta_{k3}+O(\delta),
$$
where $\delta_{jk}$ is Kronecker's symbol. It follows that $|\xi|=|\nabla \varphi_t(x)| \approx \delta$ whenever $a_t(x,\xi)=0$.  In addition,  if $b_t(x,\xi)=0$, then 
$$
\xi_3\p_3\varphi_t= - \xi_1\p_1\varphi_t - \xi_2\p_2\varphi_t=  O(\delta^3). 
$$ 
Since $\p_3\varphi_t \approx \delta$,  we conclude that $\xi_3=O(\delta^2)$. After substitution  into~\eqref{PB} it follows that 
$$
\{a_t,b_t\}(x,\xi)= 32 \delta^2+O(\delta^3).
$$
Hence, for $\delta\ll1$ the required inequality~\eqref{Poisson-lowerbound} is satisfied with $c=16\delta^2$. 
\end{proof}

\addcontentsline{toc}{section}{References}
\newcommand{\etalchar}[1]{$^{#1}$}
\def\cprime{$'$} \def\cprime{$'$}
  \def\polhk#1{\setbox0=\hbox{#1}{\ooalign{\hidewidth
  \lower1.5ex\hbox{`}\hidewidth\crcr\unhbox0}}}
  \def\polhk#1{\setbox0=\hbox{#1}{\ooalign{\hidewidth
  \lower1.5ex\hbox{`}\hidewidth\crcr\unhbox0}}}
  \def\polhk#1{\setbox0=\hbox{#1}{\ooalign{\hidewidth
  \lower1.5ex\hbox{`}\hidewidth\crcr\unhbox0}}} \def\cprime{$'$}
  \def\polhk#1{\setbox0=\hbox{#1}{\ooalign{\hidewidth
  \lower1.5ex\hbox{`}\hidewidth\crcr\unhbox0}}} \def\cprime{$'$}
  \def\cprime{$'$} \def\cprime{$'$} \def\cprime{$'$}
\providecommand{\bysame}{\leavevmode\hbox to3em{\hrulefill}\thinspace}
\providecommand{\MR}{\relax\ifhmode\unskip\space\fi MR }
\providecommand{\MRhref}[2]{%
  \href{http://www.ams.org/mathscinet-getitem?mr=#1}{#2}
}
\providecommand{\href}[2]{#2}

\end{document}